\documentclass[11pt]{amsart}
\usepackage[latin1]{inputenc}
\usepackage{amsmath,amsthm}
\usepackage{amssymb,amsfonts}
\usepackage{hyperref}
\usepackage{bbm}

\usepackage{nameref}

\usepackage{tikz}
\usepackage{tikz-qtree}
\usepackage{prftree}

\usepackage[shortlabels]{enumitem}
\usepackage{bm}
\usepackage{cite}
\usepackage{graphicx}
\usepackage{verbatim}
\usepackage{setspace}
\usepackage{color}

\usepackage[a4paper]{geometry}

\usepackage{stmaryrd}

\usepackage[mode=multiuser,status=draft,lang=english]{fixme}
\fxsetup{theme=colorsig}

\FXRegisterAuthor{M}{aM}{Matteo}
\FXRegisterAuthor{J}{aJ}{Juan}

\makeatletter
\let\orgdescriptionlabel\descriptionlabel
\renewcommand*{\descriptionlabel}[1]{%
  \let\orglabel\label
  \let\label\@gobble
  \phantomsection
  \edef\@currentlabel{#1}%
  \let\label\orglabel
  \orgdescriptionlabel{#1}%
}
\makeatother

\title[]{Boolean valued semantics for infinitary logics}
\author{Juan M. Santiago Su\'arez \& Matteo Viale}

\thanks{
The first author acknowledges support from INDAM through GNSAGA and from the project:
\emph{PRIN 2017-2017NWTM8R
Mathematical Logic: models, sets, computability.}
We thank Boban Velickovic (who outlined us the relevance of consistency properties in the analysis of forcing), and Ben de Bondt (for many useful comments).}

\theoremstyle{plain}

	\newtheorem{theorem}{Theorem}[section]
	\newtheorem{proposition}[theorem]{Proposition}
	\newtheorem{lemma}[theorem]{Lemma}
	\newtheorem{corollary}[theorem]{Corollary}
	\newtheorem{fact}[theorem]{Fact}
	
	\newtheorem{conjecture}[theorem]{Conjecture}
	\newtheorem{question}[theorem]{Question}
	\newtheorem{claim}{Claim}
	
\theoremstyle{definition}

	\newtheorem{definition}[theorem]{Definition}
	\newtheorem{notation}[theorem]{Notation}

\theoremstyle{remark}
	\newtheorem{remark}[theorem]{Remark}

\newcommand{\ZFC}{\ensuremath{\mathsf{ZFC}}}

\DeclareMathOperator{\Coll}{Coll}

\DeclareMathOperator{\RO}{RO}

\newcommand{\bool}[1]{\mathsf{#1}}

\newcommand{\Reg}[1]{\text{Reg}\left(#1\right)}

\newcommand{\Qp}[1]{\left\llbracket #1 \right\rrbracket}
\newcommand{\ap}[1]{\langle #1 \rangle}
\newcommand{\bp}[1]{\left\lbrace #1 \right\rbrace}

\newcommand{\rao} {\rightarrow}
\newcommand{\lrao} {\leftrightarrow}

\newcommand{\Rao} {\Rightarrow}

\newcommand{\MM}{\ensuremath{\text{{\sf MM}}}}

\begin{document}

\begin{abstract}
It is well known that the completeness theorem for $\mathrm{L}_{\omega_1\omega}$ fails with respect to Tarski semantics. Mansfield showed that it holds for $\mathrm{L}_{\infty\infty}$  if one replaces Tarski semantics with Boolean valued semantics. We use forcing to improve his result in order to obtain a stronger form of Boolean completeness (but only for $\mathrm{L}_{\infty\omega}$). Leveraging on our completeness result, we establish the Craig interpolation property and a strong version of the omitting types theorem for $\mathrm{L}_{\infty\omega}$ with respect to Boolean valued semantics.
We also show that a weak version of these results holds for $\mathrm{L}_{\infty\infty}$ (if one leverages instead on Mansfield's completeness theorem).
Furthermore we bring to light (or in some cases just revive) several connections between the infinitary logic $\mathrm{L}_{\infty\omega}$ and the forcing method in set theory.
%
\end{abstract}

\maketitle


\tableofcontents



\section{Introduction}\label{Section0}

This paper revives and brings to light several connections existing between infinitary logics and forcing.
The main objective of the paper is to show that boolean valued semantics is a right semantics for infinitary logics, more precisely: the class of boolean valued models with the mixing property (e.g. sheaves on compact extremally disconnected spaces by the results of \cite{PIEVIA19}) provides a complete semantics for $\mathrm{L}_{\infty\omega}$ with respect to the natural sequent calculus for infinitary logics (obtained by trivially adapting to this logic the  inference rules of Gentzen's sequent calculus for first order logic, see Section \ref{subsec:gentzencalc} below)\footnote{We note that Mansfield \cite{MansfieldConPro} proves that the larger class of boolean valued models  (e.g. presheaves on compact extremally disconnected spaces by the results of \cite{PIEVIA19}) gives a complete semantics for the logic $\mathrm{L}_{\infty\infty}$. Our completeness result is weaker than Mansfield's (as it applies only to $\mathrm{L}_{\infty\omega}$) but also stronger than his (as it provides completeness with respect to a much better behaved class of models, e.g. sheaves instead of presheaves on compact extremally disconnected topological spaces).}. Leveraging on our completeness result we are able to prove the natural form of Craig's interpolation theorem for our deductive system for $\mathrm{L}_{\infty \omega}$,
as well as a natural generalization to 
$\mathrm{L}_{\infty\omega}$ with respect to boolean valued semantics of the standard omitting types theorems which can be proved for first order logic with respect to Tarski semantics. We are also able to prove weaker forms of these results for $\mathrm{L}_{\infty\infty}$, in this latter case appealing to a completeness result of Mansfield.

A central role in our analysis of $\mathrm{L}_{\infty\omega}$ is played by the notion of consistency property.
Roughly a consistency property for a signature $\tau$ is a partial order whose elements are consistent families of infinitary $\tau$-formulae ordered by reverse inclusion. The clauses for being a consistency property in signature $\tau$ grant that a generic filter for such a forcing notion produces a maximal set of consistent $\tau$-formulae, which then can be turned into a Tarski $\tau$-structure (a term model) realizing each of them. However generic filters do not exist in  the standard universe of set theory $V$, hence such Tarski $\tau$-structures do not exist in $V$ as well, but just in a generic extension of $V$; on the other hand their semantics can be instead described in $V$ by means of boolean valued models, e.g. forcing.

Keeping in mind this idea we can show that: 
\begin{itemize}
\item
any forcing notion is forcing equivalent to a consistency property for $\mathrm{L}_{\infty\omega}$;
\item
every consistency property defines an ``elementary class'' of boolean valued models for $\mathrm{L}_{\infty\omega}$ and conversely;
\item 
most of the standard results for first order logic transfer to infinitary logic if we replace Tarski semantics with boolean valued semantics; e.g. in this paper we show that this is the case for the completeness theorem, Craig's interpolation, Beth definability, the omitting types theorem (on the other hand we can show that compactness fails for boolean valued  semantics also for $\mathrm{L}_{\infty\omega}$).
\end{itemize}

Some caveats and further comments are in order.
\begin{itemize}
\item
Sections \ref{ForConPro} and \ref{sec:for=conprop} require a basic familiarity with the forcing method (at the level of Kunen's book \cite{KUNEN}). The rest of the paper can be read by people with a loose or null knowledge of the forcing method. 
\item
Most of our results generalize to 
$\mathrm{L}_{\infty\omega}$ (and in some cases also to $\mathrm{L}_{\infty\infty}$) with respect to boolean valued semantics, results and proofs that Keisler obtains for 
$\mathrm{L}_{\omega_1\omega}$ with respect to Tarski semantics \cite{KeislerInfLog}. 
Roughly Keisler's proofs are divided in two parts: the first designs a suitable countable consistency property associated to a given countable $\mathrm{L}_{\omega_1\omega}$-theory $T$ of interest; the second
appeals to Baire's category theorem taking advantage of the considerations to follow.

 Consistency properties are designed in order that the Tarski structure induced by a maximal filter $F$ on them (seen as partial orders) realizes a certain formula $\phi$ if and only if $F$ meets a dense set  $D_\phi$ associated to $\phi$. 
If one focuses on countable theories $T$ for $\mathrm{L}_{\omega_1\omega}$, one can appeal to Baire's category theorem to find a maximal filter $F$ for the associated consistency property: $F$ meets the countable family of dense sets associated to the formulae in $T$.
This is what Keisler's proofs usually do.
  
However, if one considers an arbitrary $\mathrm{L}_{\infty\omega}$-theory $T$, one could drop the use of Baire's category theorem and replace it by describing (using forcing) as a boolean valued model the Tarski structure that Keisler's method would produce in a forcing extension where $T$ becomes a countable $\mathrm{L}_{\omega_1\omega}$-theory.
This is what we will do here.
\item
One has to pay attention to our formulation of Craig's interpolation property (e.g. Thm. \ref{thm:craigint}). We prove our result with respect to the natural deduction calculus for $\mathrm{L}_{\infty\omega}$; for this calculus it is known that the completeness theorem with respect to Tarski semantics fails (we give an explicit counterexample in Fact \ref{fac:tarskiinc}). It is no surprise hence that the semantic version of Craig's interpolation fails as well with respect to Tarski semantics (see \cite[Thm. 3.2.4]{MalitzThesis}).
On the other hand Malitz has proved an interpolation theorem for $\mathrm{L}_{\infty\infty}$ with respect to Tarski semantics using another deductive system for $\mathrm{L}_{\infty\omega}$ (introduced by Karp) which is complete for Tarski semantics \cite{MalitzThesis}. However we believe that our deductive system is better than Malitz's, since the notion of proof for our system is independent of the model of set theory we work with; for example our deductive system when restricted to $\mathrm{L}_{\infty\omega}$ is forcing invariant. Even more, for sets $\Gamma$, $\Delta$ of $\mathrm{L}_{\infty\omega}$-formulae in $V$, $\Gamma$ proves $\Delta$ is a provably $\Delta_1$-property in the parameters $\Gamma, \Delta$ in any model of $\ZFC$ to which $\Gamma$ and $\Delta$ belong: note that the existence of a proof is expressible by a $\Sigma_1$-statement while being true in any boolean valued model is expressible by a $\Pi_1$-statement (according to the Levy hierarchy as in \cite[Pag. 183]{JECHST}). In particular, $\Gamma$ proves $\Delta$ holds in $V$ according to our deductive system if and only if it holds in any (equivalently some) forcing extension of $V$. This fails badly for Malitz's deductive system, e.g. there is a sentence $\phi$ such that ``$\phi$ is valid according to Malitz's deductive system'' holds in some generic extension of $V$, but fails in $V$ and conversely.
\item
The fact that forcing and consistency properties are closely related concepts is implicit in the work 
of many; for 
example we believe this is behind Jensen's development of $\mathrm{L}$-forcing \cite{JENLFORC} and the spectacular proof by Asper\'o and Schindler that $\MM^{++}$ implies Woodin's axiom
 $(*)$ \cite{ASPSCH(*)} (see also \cite{viale2021proof} -which gives a presentation of their proof more in line with the spirit of this paper); it also seems clear that Keisler is to a large extent aware of this equivalence in his book on infinitary logics \cite{KeislerInfLog}, as well as Mansfield in his paper proving the completeness theorem for $\mathrm{L}_{\infty\infty}$ using boolean valued semantics \cite{MansfieldConPro}. On the other hand we have not been able to find anywhere an explicit statement that every complete boolean algebra is the boolean completion of a consistency property (e.g. Thm. \ref{thm:equivforcconsprop}) even if the proof of this theorem is rather trivial once the right definitions are given.
\item
While some of the results we present in this paper were known at least to some extent (e.g. the completeness theorem via boolean valued semantics for $\mathrm{L}_{\infty\infty}$ --- see Mansfield \cite{MansfieldConPro} and independently Karp \cite{Karp}), we believe that this paper gives a unified presentation of the sparse number of theorems connecting infinitary logics to boolean valued semantics we have been able to trace in the literature. Furthermore, we add to the known results some original contributions, e.g. Craig's interpolation property, Beth's definability property, the omitting types theorem, the equivalence of forcing with consistency properties, the completeness theorem for $\mathrm{L}_{\infty\omega}$ with respect to to the semantics produced by sheaves on compact extremally disconnected spaces.
\end{itemize}

The paper is organized as follows:
\begin{itemize}
\item \ref{sec:inflog} introduces the basic definitions for the infinitary logics $\mathrm{L}_{\kappa\lambda}$, including their boolean valued semantics and a Gentzen's style proof system for them.
\item \ref{sec:mainmodthres} states the main model theoretic results we obtain for $\mathrm{L}_{\infty\omega}$ and $\mathrm{L}_{\infty\infty}$.
\item \ref{sec:consprop} introduces the key notion of consistency property on which we leverage to prove all the main results of the paper.
\item \ref{ForConPro} shows that we can use consistency properties to produce boolean valued models with the mixing property (e.g. sheaves on extremally disconnected compact spaces) for any consistent $\mathrm{L}_{\infty\omega}$ theory.
\item \ref{sec:mansmodexthm} gives a proof rephrased in our terminology of the main technical result of Mansfield on this topic, e.g. that any consistency property for $\mathrm{L}_{\infty\infty}$ gives rise to a corresponding 
boolean valued model (which however may not satisfy the mixing property). 
\item \ref{sec:proofmodthres} leverages on \ref{ForConPro} and \ref{sec:mansmodexthm} to prove the theorems stated in 
\ref{sec:mainmodthres}.
\item \ref{sec:for=conprop} shows that any forcing notion can be presented as the boolean completion of a consistency property for $\mathrm{L}_{\infty\omega}$.
\item The Appendix \ref{sec:app} collects some counterexamples to properties which do not transfer from first order logic to infinitary logics (for example the failure of boolean compactness), as well as the proof of some basic facts regarding boolean valued models.
\item We close the paper with a brief list of open problems and comments.
\end{itemize}

\section{The infinitary logics $\mathrm{L}_{\kappa \lambda}$}\label{sec:inflog}

The set of formulae for a language in first order logic is constructed by induction from atomic formulae by taking negations, finite conjunctions and finite quantifications. 
$\mathrm{L}_{\kappa \lambda}$ generalizes both ``finites" to cardinals $\kappa$ and $\lambda$ allowing disjunctions and conjunctions of size less than $\kappa$ and simultaneous universal quantification of a string of variables of size less than $\lambda$.
Our basic references on this topic is V{\"a}{\"a}n{\"a}nen's book \cite{ModelsGames}.
To simplify slightly our notation we confine our attention to relational languages, i.e.  languages that do not have function symbols\footnote{With some notational efforts which we do not spell out all our results transfer easily to arbitrary signatures}.
Also, when interested in logics with quantification of infinite strings we consider natural to include signatures containing relation symbols of infinite arity. 

\subsection{Syntax}
\begin{definition}
$\mathrm{L}$ is a relational $\lambda$-signature if it contains only relation symbols of arity less than $\lambda$ and eventually constant symbols; relational $\omega$-signatures are first order signatures without function symbols.

Fix two cardinals $\lambda, \kappa$, a set of $\kappa$ variables, $\{v_\alpha : \alpha < \kappa\}$, and consider a relational $\lambda$-signature $\mathrm{L}$. The set of terms and atomic formulae for $\mathrm{L}_{\kappa \lambda}$ is constructed in analogy to first order logic using the symbols of $\mathrm{L}	\cup\{v_\alpha : \alpha < \kappa\}$. The other $\mathrm{L}_{\kappa \lambda}$-formulae are defined by induction as follows:

\begin{itemize}
\item if $\phi$ is a $\mathrm{L}_{\kappa \lambda}$-formula, then so is $\neg \phi$;
\item if $\Phi$ is a set of $\mathrm{L}_{\kappa \lambda}$-formulae of size $< \kappa$ with free variables in the set $V=\bp{v_i:i\in I}$ for some $I\in [\kappa]^{<\lambda}$, then so are $\bigwedge \Phi$ and $\bigvee\Phi$;
\item if $V=\bp{v_i:i\in I}$ for some $I\in [\kappa]^{<\lambda}$ and $\phi$ is a $\mathrm{L}_{\kappa \lambda}$-formula, then so are $\forall V \phi$ and $\exists V \phi$.
\end{itemize}
We let $\mathrm{L}_{\infty \lambda}$ be the family of $\mathrm{L}_{\kappa \lambda}$-formulae for some $\kappa$, and $\mathrm{L}_{\infty \infty}$ be the family of  $\mathrm{L}_{\kappa \lambda}$-formulae for some $\kappa,\lambda$.
\end{definition}  

The restriction on the number of free variables for the clauses $\bigwedge$ and $\bigvee$ is intended to avoid formulae for which there is no quantifier closure. Another common possibility is to call pre-formula any ``formula", and formula the ones that verify this property. 

\subsection{Boolean valued semantics}

Let us recall the following basic facts about partial orders and their Boolean completions:

\begin{definition}
Given a Boolean algebra $\bool{B}$ and a partial order $\mathbb{P} = (P,\leq)$:
\begin{itemize}
\item
$\bool{B}^+$ denotes the partial order given by its positive elements 
and ordered by $a\leq_{\bool{B}} b$ if $a\wedge b=a$.  
\item
$\bool{B}$ is $<\lambda$-complete if any subset of $\bool{B}$ of size less than $\lambda$
has an infimum and a supremum according to  $\leq_{\bool{B}}$.
\item 
A set $G \subset P$ is a prefilter if for any $a_1,\ldots,a_n \in G$ we can find $b \in G$, $b \leq a_1,\ldots,a_n$.
\item 
A set $F \subset P$ is a filter if it is a prefilter and is upward close:
\[
(a \in F \wedge a \leq b) \Rightarrow b \in F.
\]

\end{itemize}
\end{definition}

\begin{remark} \label{rema1}
Given a partial order $\mathbb{P} = (P,\leq)$: 
\begin{itemize}
\item The order topology on $P$ is the one whose open sets are given by the downward closed subsets of $P$; the sets $N_p = \bp{q\in P: q\leq p}$ form a basis for this topology.
\item $\RO(P)$ is the complete Boolean algebra given by the regular open sets of the order topology on $P$.
\item The map $p\mapsto \Reg{N_p}$ defines an order and incompatibility preserving map of $P$ into a dense subset of $(\RO(P)^+,\subseteq)$; hence $(P,\leq)$ and $(\RO(P)^+,\subseteq)$ are equivalent forcing notions.
\end{itemize}

If $\bool{B}$ is a Boolean algebra, $\bool{B}^+$ sits inside its Boolean completion  $\RO(\bool{B}^+)$
as a dense subset via the map $b\mapsto N_b$ (e.g. for all $A\in\RO(\bool{B}^+)$ there is $b\in \bool{B}$ such that $N_b\subseteq A$). 

From now on we identify $\bool{B}$ with its image in $\RO(\bool{B}^+)$ via the above map.
\end{remark}

\begin{definition} Let $\mathrm{L}$ be a relational $\lambda$-signature and $\mathsf{B}$ a $<\lambda$-complete Boolean algebra. A $\mathsf{B}$-valued model $\mathcal{M}$ for $\mathrm{L}$ is given by:
\begin{enumerate}
\item a non-empty set $M$;
\item the Boolean value of equality,
\begin{align*}
M^2 &\rao \mathsf{B} \\
(\tau,\sigma) &\mapsto \Qp{\tau=\sigma}^{\mathcal{M}}_\mathsf{B};
\end{align*} 
\item the interpretation of relation symbols $R \in \mathrm{L}$ of arity $\alpha<\lambda$ by maps
\begin{align*}
M^\alpha &\rao \mathsf{B} \\
(\tau_i:i\in \alpha) &\mapsto \Qp{R (\tau_i:i\in \alpha) }^{\mathcal{M}}_\mathsf{B};
\end{align*}
\item the interpretation $c^\mathcal{M} \in M$ of constant symbols $c$ in $\mathrm{L}$.
\end{enumerate}

We require that the following conditions hold:
\begin{enumerate}[(A)]
\item For all $\tau,\sigma,\pi \in M$,
\begin{gather*}
\Qp{\tau=\tau}^{\mathcal{M}}_\mathsf{B} = 1_\mathsf{B}, \\
\Qp{\tau=\sigma}^{\mathcal{M}}_\mathsf{B} = \Qp{\sigma=\tau}^{\mathcal{M}}_\mathsf{B}, \\
\Qp{\tau=\sigma}^{\mathcal{M}}_\mathsf{B} \wedge \Qp{\sigma=\pi}^{\mathcal{M}}_\mathsf{B} \leq \Qp{\tau=\pi}_\mathsf{B}^{\mathcal{M}}.
\end{gather*}
\item \label{eqn:subslambda}
If $R \in \mathrm{L}$ is an $\alpha$-ary relation symbol, for all $(\tau_i:\,i<\alpha), (\sigma_i:\,i<\alpha) \in M^\alpha$,
\begin{equation*}
\bigg(\bigwedge_{i\in\alpha}\Qp{\tau_i=\sigma_i}^{\mathcal{M}}_\mathsf{B} \bigg) \wedge \Qp{R(\tau_i:\,i<\alpha)}^{\mathcal{M}}_\mathsf{B} \leq \Qp{R(\sigma_i:\,i<\alpha)}^{\mathcal{M}}_\mathsf{B}.
\end{equation*}
\end{enumerate}
\end{definition}



\begin{definition}\label{def:boolvalsem} 
Fix $\mathsf{B}$ a $<\lambda$-complete Boolean algebra and $\mathcal{M}$ a $\mathsf{B}$-valued structure for a relational $\lambda$-signature $\mathrm{L}$. We define the $\RO(\mathsf{B}^+)$-value of an $\mathrm{L}_{\infty\infty}$-formula $\phi(\overline{v})$ with assignment $\overline{v} \mapsto \overline{m}$ by induction as follows:

\begin{gather*}
\Qp{R(t_i:i\in\alpha)[\overline{v} \mapsto \overline{m}]}^\mathcal{M}_{\RO(\bool{B}^+)} = \Qp{R(t_i[\overline{v} \mapsto \overline{m}]:i\in \alpha)}^\mathcal{M}_\mathsf{B} \text{ for $R\in\mathrm{L}$ of arity $\alpha<\lambda$},\\
\Qp{(\neg \phi)[\overline{v} \mapsto \overline{m}]}^\mathcal{M}_{\RO(\bool{B}^+)}  = \neg \Qp{\phi[\overline{v} \mapsto \overline{m}]}^\mathcal{M}_{\RO(\bool{B}^+)} ,\\
\Qp{(\bigwedge \Phi)[\overline{v} \mapsto \overline{m}]}^\mathcal{M}_{\RO(\bool{B}^+)}  = \bigwedge_{\phi \in \Phi} \Qp{\phi[\overline{v} \mapsto \overline{m}]}^\mathcal{M}_{\RO(\bool{B}^+)} ,\\
\Qp{(\bigvee \Phi)[\overline{v} \mapsto \overline{m}]}^\mathcal{M}_{\RO(\bool{B}^+)}  = \bigvee_{\phi \in \Phi} \Qp{\phi[\overline{v} \mapsto \overline{m}]}^\mathcal{M}_{\RO(\bool{B}^+)} ,\\
\Qp{(\forall V \phi)[\overline{v} \mapsto \overline{m}]}^\mathcal{M}_{\RO(\bool{B}^+)}  = \bigwedge_{\overline{a} \in M^V} \Qp{\phi[\overline{v} \mapsto \overline{m}, V \mapsto \overline{a}]}^\mathcal{M}_{\RO(\bool{B}^+)} ,\\ 
\Qp{(\exists V \phi)[\overline{v} \mapsto \overline{m}]}^\mathcal{M}_{\RO(\bool{B}^+)}  = \bigvee_{\overline{a} \in M^V} \Qp{\phi[\overline{v} \mapsto \overline{m}, V \mapsto \overline{a}]}^\mathcal{M}_{\RO(\bool{B}^+)}.
\end{gather*}

A $\bool{B}$-valued model is well behaved\footnote{We believe this is the right generalization that should become standard in future papers.} for $\mathrm{L}_{\kappa\lambda}$ 
if $\Qp{\phi(t_i:i\in\alpha)[\overline{v} \mapsto \overline{m}]}^\mathcal{M}_{\RO(\bool{B}^+)}\in\bool{B}$ for  any $\mathrm{L}_{\kappa\lambda}$ formula $\phi(\overline{v})$.

Let $T$ be an $\mathrm{L}_{\infty \infty}$ theory and $\mathcal{M}$ be a well behaved $\bool{B}$-valued $\mathrm{L}$-structure. The relation 
\[
\mathcal{M} \vDash T
\]
holds if 
\[
\Qp{\bigwedge T}_\bool{B}^\mathcal{M} = 1_\bool{B}.
\]
\end{definition}

Note that if $\bool{B}$ is complete any $\bool{B}$-valued model is well behaved.
We feel free to write just $\Qp{\phi(\tau_i:\,i<\alpha)}$ or $\Qp{\phi(\tau_i:\,i<\alpha)}^{\mathcal{M}}$ or $\Qp{\phi(\tau_i:\,i<\alpha)}_\mathsf{B}$ when no confusion arises on which structure we are considering or in which Boolean algebra we are evaluating the predicate $R$.

A key (but not immediately transparent) observation is that for any $\lambda$-signature $\mathrm{L}$, any  
well behaved $\bool{B}$-valued model $\mathcal{M}$ for $\mathrm{L}$ satisfies \ref{eqn:subslambda} with $R$ replaced by any $\mathrm{L}_{\infty\infty}$-formula. More precisely the following holds:
\begin{fact} \label{fac:pressubslambdaanyform}
Let $\mathrm{L}$ be a $\lambda$-relational signature and $\bool{B}$ a $<\lambda$-complete Boolean algebra.
Then for any $\bool{B}$-valued model $\mathcal{M}$ for $\mathrm{L}$,
any $\mathrm{L}_{\infty\infty}$-formula $\phi(x_i:i<\alpha)$ in displayed free variables, and any sequence $(\sigma_i:i<\alpha)$, $(\tau_i:i<\alpha)$ in
$\mathcal{M}^\alpha$
\begin{equation}\label{eqn:subslambda1}
\bigg(\bigwedge_{i\in\alpha}\Qp{\tau_i=\sigma_i}^\mathcal{M}_{\RO(\mathsf{B})^+} \bigg) \wedge \Qp{\phi(\tau_i:\,i<\alpha)}^\mathcal{M}_{\RO(\mathsf{B})^+}  \leq \Qp{\phi(\sigma_i:\,i<\alpha)}^\mathcal{M}_{\RO(\mathsf{B})^+}.
\end{equation}
\end{fact}

We prove this in Section \ref{subsec:mixfull}. 

\begin{definition} Let $\bool{B}$ be a complete Boolean algebra and
$\mathcal{M}$ a well behaved $\mathsf{B}$-valued model for some $\lambda$-signature 
$\mathrm{L}$. 
$\mathcal{M}$ has the mixing property if for any antichain $A \subset \mathsf{B}$ and $\{\tau_a : a \in A\} \subset M$ there is some $\tau \in M$ such that $a \leq \Qp{\tau=\tau_a}_\mathsf{B}$ for all $a \in A$.
\end{definition}

\begin{definition} 
Let $\lambda \leq \kappa$ be infinite cardinals,
$\bool{B}$ be a $<\lambda$-complete Boolean algebra, and
$\mathcal{M}$ be a well behaved $\mathsf{B}$-valued model for $\mathrm{L}_{\kappa\lambda}$. 

$\mathcal{M}$ is full for the logic $\mathrm{L}_{\kappa \lambda}$ if for every $\mathrm{L}_{\kappa,\lambda}$-formula $\phi(\overline{v},\overline{w})$ and $\overline{m} \in M^{\overline{w}}$ there exists $\overline{n} \in M^{\overline{v}}$ such that
\[ 
\Qp{\exists \overline{v} \phi(\overline{v},\overline{m})}_\mathsf{B} = \Qp{\phi(\overline{n},\overline{m})}_\mathsf{B}.
\]
\end{definition}

\begin{proposition}\label{prop:mixfull}
Let $\mathrm{L}$ be a $\lambda$-relational signature and $\bool{B}$ a complete Boolean algebra.
Any $\mathsf{B}$-valued model for $\mathrm{L}$ with the mixing property is full for 
$\mathrm{L}_{\infty \infty}$.
\end{proposition}
The proof of this proposition is deferred to Section \ref{subsec:mixfull}.

\begin{definition} \label{def:boolvalmod}
Let $\mathsf{B}$ be a $<\lambda$-complete Boolean algebra, $\mathcal{M}$ a full $\mathsf{B}$-valued model for $\mathrm{L}_{\kappa\lambda}$ where $\mathrm{L}$ is a relational $\lambda$-signature, and $F \subset \mathsf{B}$ a $<\lambda$-complete filter. The quotient of $\mathcal{M}$ by $F$ is the $\mathrm{L}$-structure $\mathcal{M}/_F$ defined as follows:
\begin{enumerate}
\item its domain $M/_F$ is the quotient of $M$ by the equivalence 
\[
\tau \equiv_F \sigma \lrao \Qp{\tau=\sigma} \in F,
\]
\item if $R \in \mathrm{L}$ is an $\alpha$-ary relation symbol, 
\[
R^{\mathcal{M}/_F}= \{({[\tau_i]}_F:i<\alpha) \in (M/_F)^\alpha : \Qp{R(\tau_i:i<\alpha)} \in F \},
\]
\item if $c \in \mathrm{L}$ is a constant symbol, 
\[
c^{\mathcal{M}/_F}= \bigl[c^\mathcal{M}\bigr]_F \in M/F.
\]
\end{enumerate}
\end{definition}

\begin{remark}
If  $\mathcal{M}$ a $\mathsf{B}$-valued model for $\mathrm{L}_{\kappa\lambda}$ so is
 $\mathcal{M}/_F$ is for $\bool{B}/_F$: condition \ref{eqn:subslambda} of Def. \ref{def:boolvalmod}
is satisfied by the quotient structure $\mathcal{M}/_F$ appealing to the $<\lambda$-completeness of $F$.
All other conditions of Def. \ref{def:boolvalmod} holds for $\mathcal{M}/_F$ 
just assuming $F$ being a filter.
Furthermore if $\mathcal{M}$ is full for $\mathrm{L}_{\kappa\lambda}$ 
and $F$ is also $<\kappa$-complete, so is $\mathcal{M}/_F$ (appealing to the $<\kappa$-completeness of $F$ to handle infinitary disjunctions and conjunctions and to the $<\lambda$-completeness of $F$ to handle infinitary quantifiers).
\end{remark}

\begin{theorem}[\L o\'s] \label{thm:fullLos}
Let $\lambda \leq \kappa$ be infinite cardinals, $\bool{B}$ be a $<\lambda$-complete Boolean algebra, $\mathcal{M}$ an $\mathrm{L}_{\kappa \lambda}$-full $\mathsf{B}$-valued model for $\mathrm{L}_{\kappa\lambda}$, and $U \subset \mathsf{B}$ a $<\max\bp{\kappa,\lambda}$-complete ultrafilter. Then, for every $\mathrm{L}_{\kappa\lambda}$-formula $\phi(\overline{v})$ and $\overline{\tau} \in M^{|\overline{v}|}$,
\[
\mathcal{M}/_U \vDash \phi(\overline{{[\tau]}_U}) \iff \Qp{\phi(\overline{\tau})}_\mathsf{B} \in U.
\]
\end{theorem}
\begin{proof}
A proof of the Theorem for $\mathrm{L}_{\omega\omega}$ for $\omega$-relational signatures is given in \cite[Thm. 5.3.7]{viale-notesonforcing}. The general case uses the $<\kappa$-completeness of the ultrafilter to handle  $<\kappa$-sized disjunctions and conjunctions, and its $<\lambda$-completeness and the fullness of $\mathcal{M}$ to handle quantifiers on infinite strings.
\end{proof}

From now on we will work only with complete Boolean algebras $\bool{B}$, hence 
$\bool{B}$-valued models are automatically well behaved for $\mathrm{L}_{\infty\infty}$.

\subsection{Boolean satisfiability}


\begin{definition}
$\mathrm{BVM}$ denotes the class of Boolean valued models with values on a complete Boolean algebra and $\mathrm{Sh}$ the subclass of Boolean valued models with values on a complete Boolean algebra which have the mixing property. Let $\Gamma$ and $\Delta$ be sets of $\mathrm{L}_{\infty\infty}$-formulae. In case $\Gamma = \emptyset$ we let  
\[
\Qp{\bigwedge \Gamma}_\bool{B}^\mathcal{M} = 1_{\bool{B}},
\]
and if $\Delta = \emptyset$ we let
\[
\Qp{\bigvee \Delta}_\bool{B}^\mathcal{M} = 0_\bool{B}.
\]

\begin{itemize}
\item $\Gamma$ is \emph{weakly Boolean satisfiable} if there is a complete Boolean algebra $\bool{B}$ and 
a $\bool{B}$-valued model $\mathcal{M}$ such that $\Qp{\phi}^{\mathcal{M}}_\bool{B}>0_\bool{B}$ for each $\phi\in \Gamma$.
\item $\Gamma$ is \emph{Boolean satisfiable} if there is a complete Boolean algebra 
$\bool{B}$ and 
a $\bool{B}$-valued model $\mathcal{M}$ such that $\Qp{\phi}^{\mathcal{M}}_\bool{B}=1_\bool{B}$ for each $\phi\in \Gamma$.
\item $\Gamma\vDash_\mathrm{BVM} \Delta$ if 
\[
\Qp{\bigwedge\Gamma}^{\mathcal{M}}_\bool{B}\leq\Qp{\bigvee\Delta}^{\mathcal{M}}_\bool{B}
\]
for any complete Boolean algebra $\bool{B}$ and 
$\bool{B}$-valued model $\mathcal{M}$.
\item $\Gamma\vDash_\mathrm{Sh} \Delta$ if 
\[
\Qp{\bigwedge\Gamma}^{\mathcal{M}}_\bool{B}\leq\Qp{\bigvee\Delta}^{\mathcal{M}}_\bool{B}
\]
for any complete Boolean algebra $\bool{B}$ and 
$\bool{B}$-valued model $\mathcal{M}$ with the mixing property.
\item $\Gamma \equiv_{\mathrm{BVM}} \Delta$ if $\Gamma \vDash_\mathrm{BVM} \Delta$ and $\Delta \vDash_\mathrm{BVM} \Gamma$.
\item $\Gamma \equiv_{\mathrm{Sh}} \Delta$ if $\Gamma \vDash_\mathrm{Sh} \Delta$ and $\Delta \vDash_\mathrm{Sh} \Gamma$.
\end{itemize}
\end{definition}

\subsection{Proof systems for $\mathrm{L}_{\infty\infty}$}\label{subsec:gentzencalc}

We present a proof system for $\mathrm{L}_{\infty\infty}$ that is a direct generalization of the Sequent Calculus from first order logic. $\Gamma$,$\Gamma'$,$\Delta$ and $\Delta'$ denote sets of $\mathrm{L}_{\infty\infty}$-formulae of any cardinality, $\overline{v},\overline{w}$ denote 
set-sized sequences of variables,  $\overline{t},\overline{u}$ denote set-sized sequences of terms, and $I$ denotes an index set. When dealing with sequents, and in order to make proofs shorter, we will assume that formulae only contain $\neg,\bigwedge$ and $\forall$ as logical symbols; this is not restrictive as all reasonable semantics for these logics (among which all those we consider in this paper) should validate the natural logical equivalences $\neg\forall\vec{v}\neg\phi\equiv\exists\vec{v}\phi$,  $\neg\bigwedge_{i\in I}\neg\phi_i\equiv\bigvee_{i\in I}\phi_i$.

\begin{definition} Given $\Gamma,\Delta$ arbitrary sets of $\mathrm{L}_{\infty\infty}$-formulae, a proof of $\Gamma \vdash \Delta$ in $\mathrm{L}_{\infty \infty}$ is a sequence $(s_\alpha)_{\alpha \leq \beta}$ of sequents, where  $s_\beta$ is $\Gamma \vdash \Delta$ and each element $s_\alpha$ is either an axiom or comes from an application of the following rules to $(s_i)_{i<\alpha}$.
\end{definition}

\begin{displaymath}
  \begin{array}{lcc@{\qquad}l}
    \mbox{Axiom rule} &
    \prftree{}{\Gamma, \phi \vdash \phi, \Delta} &
	\prftree{\Gamma, \phi \vdash \Delta}{\Gamma' \vdash \phi, \Delta'}{\Gamma,\Gamma' \vdash \Delta, \Delta'} &   
    \mbox{Cut Rule} \\
        	         \\
    \mbox{Substitution} &
    \prftree{\Gamma \vdash \Delta}{\Gamma(\overline{w} \diagup \overline{v}) \vdash \Delta(\overline{w} \diagup \overline{v})} & 
    \prftree{\Gamma\vdash \Delta}{\Gamma,\Gamma' \vdash \Delta, \Delta'} &   
    \mbox{Weakening}\\
        			\\
    \mbox{Left Negation} &
    \prftree{\Gamma \vdash \phi, \Delta}{\Gamma, \neg \phi \vdash \Delta} & 
    \prftree{\Gamma, \phi \vdash \Delta}{\Gamma \vdash \neg \phi, \Delta} &
    \mbox{Right Negation} \\
    					  \\
    \mbox{Left Conjunction} &
    \prftree{\Gamma,\Gamma' \vdash \Delta}{\Gamma,\bigwedge \Gamma' \vdash \Delta} &
    \prftree{\Gamma \vdash \phi_i, \Delta \ ,\ i \in I}{\Gamma \vdash \bigwedge_{i \in I} \{\phi_i : i \in I\}, \Delta} &
    \mbox{Right Conjunction} \\
    						 \\
    \mbox{Left Quantification} &
    \prftree{\Gamma, \phi(\overline{t} \diagup \overline{v}) \vdash \Delta}{\Gamma, \forall \overline{v} \phi(\overline{v}) \vdash \Delta} &
    \prftree[r]{*}{\Gamma \vdash \phi(\overline{w} \diagup \overline{v} ), \Delta}{\Gamma \vdash \forall \overline{v}\phi(\overline{v}), \Delta} &
    \mbox{Right Quantification}\\
    						   \\
    \mbox{Equality 1} &
    \prftree{}{ v_\alpha = v_\beta \vdash v_\beta = v_\alpha} &
    \prftree[r]{}{}{\overline{u} = \overline{t}, \phi(\overline{t}) \vdash \phi(\overline{u})} &
    \mbox{Equality 2}\\

  \end{array}
\end{displaymath}
\vspace{0,5cm}

* The Right Quantification rule can only be applied in the case that none of the variables from $\overline{w}$ occurs free in formulae of $\Gamma\cup\Delta\cup\bp{\phi}$.\medskip

\begin{remark}
It needs to be noted that with this deduction system the completeness theorem for $\mathrm{L}_{\infty\infty}$ (even for $\mathrm{L}_{\omega_2\omega}$) fails for the usual semantics given by Tarski structures.

Remark first that our proof system is forcing invariant: the existence of a proof for a certain sentence is described by a $\Sigma_1$ statement in parameter the sequent to be proved; if the proof exists in $V$ it exists in any further extension of $V$.

Consider now  a set of $\kappa$ constants $\{c_\alpha : \alpha < \kappa\}$ for $\kappa>\omega$ and the sentence 
\[ 
\psi := \bigg( \bigwedge_{\omega \leq \alpha \neq \beta} c_\alpha \neq c_\beta \bigg) \Rao \exists v \bigg( \bigwedge_{n < \omega} v \neq c_n \bigg). 
\]
The sentence $\psi$ is valid in the usual Tarski semantics but it cannot be proved (in our deduction system or in any forcing invariant system) since the sentence is no longer valid when moving to $V[G]$ for $G$ a $V$-generic filter for $\Coll(\omega,\kappa)$. 

Malitz \cite[Thm. 3.2.4]{MalitzThesis} showed also that the above formula is a counterexample to Craig's interpolation property for Tarski semantics in $\mathrm{L}_{\infty\omega}$.
\end{remark}

Our opinion is that a proof system should not depend on the model of set theory in which one is working, which is the case for the proof system we present here at least when restricted to $\mathrm{L}_{\infty\omega}$.
In contrast with our point of view, one finds a complete proof system for Tarski semantics on $\mathrm{L}_{\infty\infty}$ in Malitz's thesis \cite[Thm. 3.3.1]{MalitzThesis}, however this proof system (which by the way is due to Karp \cite[Ch. 11]{Karp}) is not forcing invariant e.g. a proof of some sequent in some model of set theory may not be anymore a proof of that same sequent in some forcing extension.

\section{Main model theoretic results}\label{sec:mainmodthres}

These are the main model theoretic results of the paper.

\subsection{Results for $\mathrm{L}_{\infty \omega}$}

\begin{theorem}[Boolean Completeness for $\mathrm{L}_{\infty\omega}$]  \label{them:boolcompl}
Let $\mathrm{L}$ be an $\omega$-relational signature. 
The following are equivalent for $T,S$ 
sets of $\mathrm{L}_{\infty\omega}$-formulae.
\begin{enumerate}
\item \label{thm:boolcomp1}
$T\models_{\mathrm{Sh}}S$,
\item \label{thm:boolcomp2}
$T\models_{\mathrm{BVM}}S$,
\item \label{thm:boolcomp3}
$T\vdash S$.
\end{enumerate}
\end{theorem}

\begin{theorem}[Boolean Craig Interpolation]\label{thm:craigint}
Assume $\vDash_{\mathrm{Sh}}\phi \rightarrow \psi$ with $\phi,\psi \in \mathrm{L}_{\kappa\omega}$. 
Then there exists a sentence $\theta$ in $\mathrm{L}_{\kappa\omega}$ such that 

\begin{itemize}
\item $\vDash_{\mathrm{Sh}} \phi \rightarrow \theta$,
\item $\vDash_{\mathrm{Sh}} \theta \rightarrow \psi$,
\item all non logical symbols appearing in $\theta$ appear both in $\phi$ and $\psi$. 
\end{itemize}
\end{theorem}

Recall the Beth definability property:
\begin{definition}
Let $\mathrm{L}$ be a relational $\lambda$-signature and $R$ be an $\alpha$-ary relation symbol not in $\mathrm{L}$ for some $\alpha<\lambda$.
Given $\lambda,\kappa\in\mathsf{Card}\cup\bp{\infty}$,
let $T$ be a  $\mathrm{L}_{\kappa\lambda}'$-theory for $\mathrm{L}'=\mathrm{L}\cup\bp{R}$.
 
 \begin{itemize}
 \item $R$ is implicitly Boolean definable from $T$ in a relational $\lambda$-signature $\mathrm{L}$
if the following holds: whenever $\mathcal{M}$ and $\mathcal{N}$ are $\bool{B}$-valued models of $T$ with domain $M$ such that
 $\mathcal{M}\restriction\mathrm{L}=\mathcal{N}\restriction\mathrm{L}$, we have that
 $\Qp{R(\tau_i:i\in\alpha)}^{\mathcal{M}}=\Qp{R(\tau_i:i\in\alpha)}^{\mathcal{N}}$
 for all $(\tau_i:i\in\alpha)\in M^\alpha$.
  \item $R$ is explicitly Boolean definable from $T$ in $\mathrm{L}_{\kappa\lambda}$
if
\[
T\vdash \forall (v_i:i\in\alpha)\,(R(v_i:i\in\alpha)\leftrightarrow\phi(v_i:i\in\alpha))
\]
 for some $\mathrm{L}_{\kappa\lambda}$-formula $\phi(v_i:i\in\alpha)$.
 \end{itemize}
The Boolean Beth definability property for $\mathrm{L}_{\kappa\lambda}$ (with $\lambda,\kappa\in\mathsf{Card}\cup\bp{\infty}$) states that for all
relational $\lambda$-signatures
$\mathrm{L}$ and  $(\mathrm{L}\cup\bp{R})_{\infty\omega}$-theory $T$,
$R$ is implicitly definable from $T$ in $\mathrm{L}\cup\bp{R}$
if and only if it is explicitly definable from $T$ in $\mathrm{L}_{\kappa\lambda}$.
 \end{definition}
 
  This is a standard consequence of Craig's interpolation and completeness
  (see for example \cite[Thm. 6.42]{ModelsGames}; the same proof applies to our context in view of  the properties of our calculus $\vdash$).
 \begin{theorem}\label{thm:bethdef}
 $\mathrm{L}_{\infty\omega}$ has the Boolean Beth definability property.
 \end{theorem}
 
Another main result we present is the Boolean omitting types theorem. We need to clarify
 some notation so to make its statement intelligible. Suppose $\Sigma(v_1,\dots,v_n)$ is a set of $\mathrm{L}_{\infty\infty}$-formulae in free variables $v_1,\dots,v_n$. We say that a model $\mathcal{M}$ realizes $\Sigma(v_1,\dots,v_n)$ if there  exists some $m_1,\dots,m_n \in M$ such that 
\[
\mathcal{M} \vDash \bigwedge \Sigma(m_1,\dots,m_n).
\]
 $\mathcal{M}$ omits the type $\Sigma$ amounts to say that for any $m_1,\dots,m_n \in M$, 
\[
\mathcal{M} \vDash \bigvee_{\phi \in \Sigma} \neg \phi(m_1,\dots,m_n).
\]
Thus, a model $\mathcal{M}$ omits the family of types $\mathcal{F} = \{\Sigma(v_1,\ldots,v_{n_\Sigma}) : \Sigma\in \mathcal{F}\}$ if it models the sentence
\[
\bigwedge_{\Sigma\in \mathcal{F}} \forall \overline{v}_\Sigma \bigvee \bp{\neg\phi(\overline{v}_\Sigma):\phi\in \Sigma}.
\]
In the following proof the sets $\Phi$ will be playing the roles of $\bp{\neg \psi : \psi \in \Sigma}$, where $\Sigma$ is the type we wish to omit. In this context, the type $\Sigma$ is not isolated by a sentence $\theta$ if whenever there is a model of $\theta$, there is also a model of $\theta \wedge \neg\phi$ for some $\phi \in \Sigma$.

%
%
%

\begin{theorem}[Boolean Omitting Types Theorem]\label{thm:omittypthm}
Let $T$ be a set-sized Boolean satisfiable $\mathrm{L}_{\infty\omega}$-theory. Assume $\mathcal{F}$ is a set-sized family such that each $\Phi \in\mathcal{F}$ is a set of $\mathrm{L}_{\infty \omega}$-formulae with the property that each $\phi \in \Phi$ has free variables among $v_0,\dots, v_{n_\Phi - 1}$. Let $\mathrm{L}_{T,\mathcal{F}}$ be the smallest fragment of $\mathrm{L}_{\infty \omega}$ such that $T, \Phi \subset \mathrm{L}_{T,\mathcal{F}}$ for all $\Phi \in \mathcal{F}$. Suppose that no $\Phi$ is boolean isolated in $\mathrm{L}_{T,\mathcal{F}}$, i.e.; for all $\mathrm{L}_{T,\mathcal{F}}$-formula $\theta$ in free variables $v_0,\dots,v_{n_\theta - 1}$,
\[
T + \exists v_0 \dots v_{n_\theta - 1}\,\theta
\]
is Boolean satisfiable if and only if so is 
\[
T + \exists v_0 \dots v_{\max\bp{n_\theta - 1, n_\Phi - 1}}\,[\theta \wedge \phi]
\]
for some $\phi \in \Phi$. Then there exists a Boolean valued model $\mathcal{M}$ with the mixing property such that 
\[
\mathcal{M} \vDash T + \bigwedge_{\Phi\in\mathcal{F}} \forall v_0 \dots v_{n_\Phi - 1}\,  \bigvee \Phi.
\]   
\end{theorem}

\subsection{Results for $\mathrm{L}_{\infty \infty}$}

Theorem \ref{thm:complMansfield} below is due to Mansfield \cite{MansfieldConPro}.
We do not know whether any of these results hold for $\lambda$-relational signatures which are not first order.

\begin{theorem}\label{thm:complMansfield}
\cite[Thm. 1]{MansfieldConPro}
Let $\mathrm{L}$ be an $\omega$-relational signature. 
The following are equivalent for $T,S$ 
sets of $\mathrm{L}_{\infty\infty}$-formulae.
\begin{enumerate}
\item \label{thm:boolcomp2MANS}
$T\models_{\mathrm{BVM}}S$,
\item \label{thm:boolcomp3MANS}
$T\vdash S$.
\end{enumerate}
\end{theorem}

\begin{theorem}[Boolean Craig Interpolation]\label{thm:craigint2}
Let $\mathrm{L}$ be an $\omega$-relational signature. 
Assume $\vDash_{\mathrm{BVM}}\phi \rightarrow \psi$ with $\phi,\psi \in \mathrm{L}_{\kappa\lambda}$. 
Then there exists a sentence $\theta$ in $\mathrm{L}_{\kappa\lambda}$ such that 

\begin{itemize}
\item $\vDash_{\mathrm{BVM}} \phi \rightarrow \theta$,
\item $\vDash_{\mathrm{BVM}} \theta \rightarrow \psi$,
\item all non logical symbols appearing in $\theta$ appear both in $\phi$ and $\psi$. 
\end{itemize}
\end{theorem}

\begin{theorem}\label{thm:bethdef2}
$\mathrm{L}_{\infty\infty}$ has the Boolean Beth definability property.
\end{theorem}

As in the case of interpolation, one can prove a version of the omitting types theorem in $\mathrm{L}_{\infty \infty}$ where the obtained model is not mixing in general. Nonetheless, we do not think its proof nor the statement while introduce anything of relevance other than the information found in theorems \ref{thm:omittypthm} and \ref{thm:craigint2}.


\section{Consistency properties for relational $\omega$-signatures}\label{sec:consprop}

Consistency properties are partial approximations to the construction of a model of an infinitary theory.
In first order logic the main tool for constructing Tarski models of a theory is the compactness theorem. However, this technique is not suited for the infinitary logics $\mathrm{L}_{\kappa \lambda}$ since it fails even at the simplest case $\mathrm{L}_{\omega_1 \omega}$. Actually, a cardinal $\kappa$ is (weakly) compact if and only if the (weak) compactness theorem holds for the logic $\mathrm{L}_{\kappa \kappa}$. Thus a new recipe for constructing models is needed. This is given by the notion of consistency property.
Our aim is to show that by means of consistency properties one gets a powerful tool to produce boolean valued models of infinitary logic. 
We follow (generalizing it) the approach of Keisler's book \cite{KeislerInfLog} to consistency properties for
$\mathrm{L}_{\omega_1\omega}$.\vspace{0,5cm}
 
First of all it is convenient to reduce the satisfaction problem to formulae where negations occur only in atomic formulae. 
 We use the abbreviation $\vec{v}$ to denote a sequence of variables. 
 Similarly $\vec{c}$ denotes a string of constants.
 
\begin{definition} \label{MovNegIns} Let $\phi$ be a $\mathrm{L}_{\infty \infty}$-formula. 
We define $\phi \neg$ (moving a negation inside) by induction on the complexity of formulae:

\begin{itemize}
\item If $\phi$ is an atomic formula $\varphi$, $\phi \neg$ is $\neg \varphi$.
\item If $\phi$ is $\neg \varphi$, $\phi \neg$ is $\varphi$.
\item If $\phi$ is $\bigwedge \Phi$, $\phi \neg$ is $\bigvee \{\neg \varphi : \varphi \in \Phi\}$.
\item If $\phi$ is $\bigvee \Phi$, $\phi \neg$ is $\bigwedge \{\neg \varphi : \varphi \in \Phi\}$.
\item If $\phi$ is $\forall \vec{v} \varphi(\vec{v})$, $\phi \neg$ is $\exists \vec{v} \neg \varphi(\vec{v})$.
\item If $\phi$ is $\exists \vec{v} \varphi(\vec{v})$, $\phi \neg$ is $\forall \vec{v} \neg \varphi(\vec{v})$.
\end{itemize}
\end{definition}

It is easily checked that $\neg \phi$ and $\phi \neg$ are equivalent (under any reasonable notion of equivalence, e.g. boolean satisfiability or provability). This operation is used in the proof of Thm. \ref{ModExiThe}, Thm. \ref{GenFilThe} and Thm. \ref{ManModExi}.

%


\begin{definition} \label{def:ConProInf}
Let $\mathrm{L} = \mathcal{R} \cup \mathcal{D}$ be a relational $\omega$-signature where the relation symbols are in $\mathcal{R}$ and $\mathcal{D}$ is the set of constants.
Given an infinite set of constants $\mathcal{C}$ disjoint from $\mathcal{D}$,
consider $\mathrm{L}(\mathcal{C})$ the signature obtained by extending $\mathrm{L}$ with the constants in $\mathcal{C}$. A set 
$S$ whose elements are set sized subsets of $\mathrm{L}(\mathcal{C})_{\infty\infty}$ is a consistency property for $\mathrm{L}(\mathcal{C})_{\infty\infty}$ if for each 
$s \in S$ the following properties hold:


\begin{enumerate}
\item[(Con)]\label{conspropCon} for any $r\in S$ and any $\mathrm{L}(\mathcal{C})_{\infty\infty}$-sentence $\phi$ either $\phi\not\in r$ or $\neg\phi\not\in r$,
\item[(Ind.1)]\label{conspropInd1} if $\neg \phi \in s$, $s \cup \{\phi \neg\} \in S$,
\item[(Ind.2)]\label{conspropInd2} if $\bigwedge \Phi \in s$, then for any $\phi \in \Phi$, $s \cup \{\phi\} \in S$,
\item[(Ind.3)]\label{conspropInd3} if $\forall \vec{v} \phi(\vec{v}) \in s$, then for any $\vec{c} \in (\mathcal{C}\cup\mathcal{D})^{|\vec{v}|}$, $s \cup \{\phi(\vec{c})\} \in S$,
\item[(Ind.4)]\label{def:conspropInd4} if $\bigvee \Phi \in s$, then for some $\phi \in \Phi$, $s \cup \{\phi\} \in S$,
\item[(Ind.5)]\label{def:conspropInd5} if $\exists \vec{v} \phi(\vec{v}) \in s$, then for some $\vec{c} \in \mathcal{C}^{|\vec{v}|}$, $s \cup \{\phi(\vec{c})\} \in S$,
\item[(Str.1)]\label{def:conspropStr1} if $c,d \in \mathcal{C}\cup\mathcal{D}$ and $c = d \in s$, then $s \cup \{d = c\} \in S$,
\item[(Str.2)] \label{def:conspropStr2} if $c,d \in \mathcal{C} \cup \mathcal{D}$ and $\{c = d, \phi(d)\} \subset s$, then $s \cup \{\phi(c)\} \in S$,
\item[(Str.3)] \label{def:conspropStr3} if $d \in \mathcal{C} \cup \mathcal{D}$, then for some $c \in \mathcal{C}$, $s \cup \{c = d\} \in S$.
\end{enumerate} 
\end{definition}

The following result, due to Makkai \cite{MakkaiConPro}, shows the value of consistency properties for 
$\mathrm{L}_{\omega_1 \omega}$.

\begin{theorem}[Model Existence Theorem] \label{ModExiThe} Let $\mathrm{L}$ be a countable relational $\omega$-signature,  $\mathcal{C}$ a countable set of constants, and $S \subset [\mathrm{L}(\mathcal{C})_{\omega_1 \omega}]^{\leq \omega}$ be a consistency property of countable size.
Then any $s \in S$ is realized in some Tarski model for $\mathrm{L}$.
\end{theorem}

Now let us give a few examples of consistency properties for $\mathrm{L}(\mathcal{C})_{\infty\omega}$ and $\mathrm{L}(\mathcal{C})_{\infty\infty}$.

\begin{enumerate}
\item
Consider $\mathcal{K}$ a class of Tarski structures for $\mathrm{L}(\mathcal{C})$. The following families are consistency properties for $\mathrm{L}(\mathcal{C})_{\infty\infty}$: 
\begin{itemize}
\item for fixed infinite cardinals $\lambda \geq \kappa,\mu$ and $\mathcal{C}$ a set of constants of size at least $\lambda$,
\[
S_{\lambda,\kappa} = \{s \in [\mathrm{L}(\mathcal{C})_{\lambda \mu}]^{\leq \kappa} : \,\exists \mathcal{A} \in \mathcal{K},\ \mathcal{A} \vDash \bigwedge s\},
\]
\item $S_{\lambda,< \omega} = \{s \in [\mathrm{L}(\mathcal{C})_{\lambda \mu}]^{< \omega} : \,\exists  \mathcal{A} \in \mathcal{K},\ \mathcal{A} \vDash \bigwedge s\}$,
\item $S_{\lambda, \kappa}$ and $S_{\lambda,< \omega}$ where only a finite number of constants from $\mathcal{C}$ appear in each $s \in S$.  
\end{itemize}
 
\item\label{exm:BvalmodConsProp}
Let $\mathcal{M}$ be a $\bool{B}$-valued model with domain $M$ for a signature $\mathrm{L}=\mathcal{R}\cup\mathcal{D}$. We let $\mathcal{C}=M$ and $S$ be the set of  finite (less than $\kappa$-sized,\dots) sets $r$ of $\mathrm{L}(M)_{\kappa\lambda}$-sentences such that 
\[
\Qp{\bigwedge r}^{\mathcal{M}}_{\bool{B}}>0_{\bool{B}}.
\]
Then $S$ is a consistency property.
 
\item
The following families are consistency properties for $\mathrm{L}(\mathcal{C})_{\infty\omega}$:
\begin{itemize}
\item Any of the previous cases where the Tarski structures in $\mathcal{K}$ may exist only in some generic extension
of $V$: e.g. given a $\mathrm{L}_{\infty\omega}$-theory $T$, $T$ may not be consistent in $V$ with respect to the Tarski semantics for 
$\mathrm{L}_{\infty\omega}$,
but  $T$ may become consistent with respect to the Tarski semantics for 
$\mathrm{L}_{\infty\omega}$ in some generic extension of $V$; one can then use the forcible properties of the Tarski models of $T$ existing in some generic extension of $V$ to define a consistency property in $V$. 
\end{itemize}
\end{enumerate}
 The last example is based on the following observation:
let $S$ be a consistency property for $\mathrm{L}_{\kappa^+ \omega}$ of size $\kappa$ 
whose elements are all sets of formulae of size at most $\kappa$ existing in $V$.
Let $G$ be a $V$-generic filter for the forcing $\Coll(\omega,\kappa)$. Then, in the generic extension $V[G]$, $S$ becomes a consistency property of countable size for $\mathrm{L}_{\omega_1^{V[G]} \omega}$ and the Model Existence Theorem \ref{ModExiThe} applied in $V[G]$ provides the desired Tarski model of any $s\in S$. 

\begin{definition}
    Suppose $\kappa$ is an infinite cardinal and let $\mathrm{L}$ be a signature. A fragment $\mathrm{L}_\mathcal{A} \subset \mathrm{L}_{\kappa \omega}$ consists in a set of $\mathrm{L}_{\kappa \omega}$-formulas such that:
    \begin{itemize}
        \item $\mathrm{L}_\mathcal{A}$ is closed under $\neg$, $\wedge$ and $\vee$,
        \item if $\phi\in \mathrm{L}_\mathcal{A}$ and $v$ is a variable appearing in some $\mathrm{L}_\mathcal{A}$-formula, $\forall v \phi$ and $\exists v \phi$ belong to $\mathrm{L}_\mathcal{A}$,
        \item $\mathrm{L}_\mathcal{A}$ is closed under subformulas,
        \item if $\phi \in \mathrm{L}_\mathcal{A}$, then $\phi \neg \in \mathrm{L}_\mathcal{A}$,
        \item if $\phi \in \mathrm{L}_\mathcal{A}$, then there is a variable appearing in $\mathrm{L}_\mathcal{A}$ which does not occur in $\phi$,
        \item if $\phi(v) \in \mathrm{L}_\mathcal{A}$ and $t$ is any $\mathrm{L}$-term, $\phi(t) \in \mathrm{L}_\mathcal{A}$,
        \item if $\phi(v_1,\ldots,v_n) \in \mathrm{L}_\mathcal{A}$ and $w_1,\ldots,w_n$ are variable appearing in $\mathrm{L}_\mathcal{A}$, $\phi(w_1,\ldots,w_n) \in \mathrm{L}_\mathcal{A}$.
    \end{itemize}    
\end{definition}

\begin{remark} 
    Suppose $\kappa$ is an infinite cardinal and let $\mathrm{L}$ be a signature. Let $T$ be a set of $\mathrm{L}_{\kappa \omega}$-formulae. Then there exists a smallest fragment $\mathrm{L}_\mathcal{A}$ such that $T \subset \mathrm{L}_\mathcal{A}$ and 
    \[
    |\mathrm{L}_\mathcal{A}| = |L| + |T| + \kappa. 
    \]
\end{remark}

\section{Forcing with consistency properties} \label{ForConPro}

In this section we assume that $\mathrm{L}$ denotes a set-sized $\omega$-relational signature,  $\mathcal{C}$ is a set of fresh constants, and $S \subset \mathcal{P}(\mathrm{L}(\mathcal{C})_{\infty \omega})$ is a set-sized consistency property. 

We start by noting the following:

\begin{remark} \label{S-PS} If $S$ is a consistency property, so is 
$\{s \subset \mathrm{L}(\mathcal{C})_{\infty \omega} : \exists s_0 \in S\, s \subseteq s_0\}$.
\end{remark}

\begin{definition} Let $S$ be a consistency property over $\mathrm{L}(\mathcal{C})_{\infty\omega}$ for a set of constants $\mathcal{C}$ and a relational $\omega$-signature $\mathrm{L}$. 
The forcing notion $\mathbb{P}_S$ is given by:

\begin{itemize}
\item domain: $\{s \subset \mathrm{L}(\mathcal{C})_{\infty \omega} : \exists s_0 \in S \,(s \subseteq s_0)\}$;
\item order: $p \leq q$ if and only if $q \subseteq p$.
\end{itemize}

Given a filter $F$ on $\mathbb{P}_S$, $\Sigma_F = \bigcup F$. 
\end{definition}

The proof of the Model Existence Theorem for $\mathrm{L}_{\omega_1 \omega}$ as given in \cite{KeislerInfLog}, corresponds naturally to the construction for a given consistency property $S$ of a suitable filter $G$ on $\mathbb{P}_S$ generic over countably many dense sets. 
The clauses of a consistency property are naturally attached to dense sets a maximal filter $G$ on $\mathbb{P}_S$ needs to meet in order to produce a Tarski model of the formulae $\phi\in \bigcup G$.
For example, suppose $\bigvee \Phi \in s_0 \in S$. Clause \ref{def:conspropInd4} together with Remark \ref{S-PS} states that the set $\{s \in S: \,\Phi\cap s\neq\emptyset\}$ is dense below $s_0$. In Keisler's case the elements of a consistency property are countable and each $\mathrm{L}(\mathcal{C})_{\omega_1 \omega}$-formula has countably many subformulae. So, one can take an enumeration of all the dense sets at issue and diagonalize.
In the general case for $\mathrm{L}_{\infty \omega}$ one deals with many more dense sets, hence a filter meeting all the relevant dense sets may not exists. However we can translate Keisler's argument using forcing and produce a Boolean valued model for the associated consistency property.
%

For the rest of this section we work with consistency properties made up from finite sets of sentences. The reader familiar with Keisler's book \cite{KeislerInfLog} will find this restriction natural.

We split our generalization of Keisler's result in two pieces. The first piece shows how far one can go in proving the Model Existence Theorem assuming only the existence of a maximal filter on a consistency property $S$. The second one shows how genericity fills the missing gaps. 

\begin{fact} Let $S$ be a consistency property for $\mathrm{L}(\mathcal{C})_{\infty\omega}$ for a set of constants $\mathcal{C}$. Assume $S$ consists only of finite sets of formulae.
Suppose $F \subseteq \mathbb{P}_S$ is a filter. Then $[\Sigma_F]^{<\omega}=F$.
\end{fact}

\begin{proof}
The inclusion $F \subset [\Sigma_F]^{< \omega}$ follows by definition of $\Sigma_F$. We now prove $[\Sigma_F]^{< \omega} \subseteq F$. Suppose $p=\bp{\phi_1,\dots,\phi_n} \in [\Sigma_F]^{< \omega}$. Then there exist $s_1,\ldots,s_n \in F$ with each $\phi_i\in s_i$. Hence $p \subseteq \bigcup_{i \leq n} s_i$. Since $F$ is a filter, we have $\bigcup_{i \leq n} s_i \in F \subseteq \mathbb{P}_S$. The set $p$ is a condition in $\mathbb{P}_S$ since $\mathbb{P}_S$ is closed under subsets. Finally, $\bigcup_{i \leq n} s_i \leq p$ and $\bigcup_{i \leq n} s_i \in F$ imply $p \in F$. 
\end{proof}

\begin{definition} \label{DefStr}
Given a relational $\omega$-signature $\mathrm{L}=\mathcal{R}\cup\mathcal{D}$,  a set of fresh constants $\mathcal{C}$,  and a consistency property $S$ for $\mathrm{L}(\mathcal{C})_{\infty\omega}$, let $F$ be
 a maximal filter for $\mathbb{P}_S$. 
 
$\mathcal{A}_F=(A_F,R_F:R\in\mathcal{R}, d_F:d\in\mathcal{D})$ is the following string of objects:
\begin{itemize}
\item
$A_F$ is the set of equivalence classes on $\mathcal{C}\cup\mathcal{D}$ for the equivalence relation
$c\cong_F d$ if and only if $(c=d)\in \Sigma_F$,
\item for $R \in \mathcal{D}$ $n$-ary relation symbol and $c_1,\ldots,c_n \in \mathcal{C} \cup \mathcal{D}$, $R_F([c_1]_F,\dots,[c_n]_F)$ holds if and only if $R(c_1,\dots,c_n)\in \Sigma_F$,
\item $d_F=[d]_F$ for any $d\in\mathcal{D}\cup\mathcal{C}$.
\end{itemize}
\end{definition}

Consistency properties are so designed that $\mathcal{A}_F$ is a Tarski structure for $\mathrm{L}(\mathcal{C})$:
\begin{fact}\label{fac:TarStrAF}
Let $\mathrm{L}=\mathcal{R}\cup\mathcal{D}$ be a relational $\omega$-signature, 
$\mathcal{C}$ a fresh set of constants,
$S$ a consistency property for $\mathrm{L}(\mathcal{C})_{\infty\omega}$, $F$ a maximal filter for $\mathbb{P}_S$. Then $\mathcal{A}_F$ is a Tarski structure for $\mathrm{L}(\mathcal{C})$.
\end{fact}

\begin{proof}
We need to check that the definition of $A_F$ and of $R_F$ does not depend on the chosen representatives $c_1,\ldots,c_n$. Suppose $c_1 = d_1, \ldots, c_n = d_n, R(c_1 \ldots c_n)\in \Sigma_F$.
By the previous Fact $\{c_1 = d_1, \ldots, c_n = d_n, R(c_1 \ldots c_n)\} \in F$. Hence by Clause
\ref{def:ConProInf}(Ind2) for any $p\supseteq \{c_1 = d_1, \ldots, c_n = d_n, R(c_1 \ldots c_n)\}$ in 
$\mathbb{P}_S$,
$p\cup\bp{R(d_1,\dots,d_n)}\in \mathbb{P}_S$. This combined with Clause \ref{def:ConProInf}(Con) gives that no $p\in \mathbb{P}_S$ can contain
$ \{c_1 = d_1, \ldots, c_n = d_n, R(c_1 \ldots c_n), \neg R(d_1,\dots,d_n)\}$.
Hence by maximality  of $F$, 
\[
\{c_1 = d_1, \ldots, c_n = d_n, R(c_1 \ldots c_n), R(d_1,\dots,d_n)\}\in F
\] must be the case.
\end{proof}

\begin{lemma} \label{MaxFilThe} 
Let $\mathrm{L}$ be a relational $\omega$-signature and $\mathcal{C}$ an infinite set of constants disjoint from $\mathrm{L}$. 
Assume $S \subset [\mathrm{L}(\mathcal{C})_{\infty \infty}]^{< \omega}$ is a consistency property.
Let $F \subseteq \mathbb{P}_S$ be a maximal filter on $\mathbb{P}_S$. 
Consider $\Sigma'_F \subset \Sigma_F$ the set of (quantifier free) formulae $\psi\in\Sigma_F$ which are either atomic, negated atomic, or such that any subformula of $\psi$ which is neither atomic nor negated atomic contains just the logical constant $\bigwedge$. 
Then $\mathcal{A}_F \vDash \Sigma'_F$.
\end{lemma}

\begin{proof}
%
%

We do it by induction on the complexity of $\psi \in \Sigma'_F$. 
First note that $\mathbb{P}_S$ is a consistency property of which $S$ is a dense subset.
The atomic case follows by Def. \ref{DefStr}. For the remaining inductive clauses we proceed as follows:

\begin{description}
\item[$\neg$] 
Suppose $\psi = \neg \phi \in \Sigma'_F$ with $\phi$ an atomic formula. Let's see that 
\[
\mathcal{A}_F \nvDash \phi.
\]
Since $\phi$ is atomic it is enough to check $\phi \notin \Sigma'_F$. Suppose otherwise. Then there exists $p \in F$ with $\phi \in p$. Also $\psi \in q$ for some $q \in F$. By compatibility of filters there exists $r \leq p,q$. But $\phi, \neg \phi \in r$ contradicts clause \ref{def:ConProInf}(Con) for $\mathbb{P}_S$. Therefore
\[ 
\mathcal{A}_F \vDash \psi.
\]
\item[$\bigwedge$] 
Suppose $\psi = \bigwedge \Phi$ is in $\Sigma_F'$. One needs to check 
\[
\mathcal{A}_F \vDash \phi
\]
for any $\phi \in \Phi$. Fix such a $\phi\in\Phi$. We start by showing that if $\bigwedge \Phi \in \Sigma'_F$, $\phi$ is also in $\Sigma'_F$. It is enough to check $\phi \in \Sigma_F$, and then apply the inductive assumptions on $\phi\in\Sigma'_F$, to get that $\mathcal{A}_F\models\phi$. 
Towards this aim we note the following:
\begin{quote}
For any $q\in \mathbb{P}_S$ with $\bigwedge\Phi\in q$,
$q\cup\bp{\phi}\in\mathbb{P}_S$, while $q\cup\mathbb{\neg\phi}\not\in\mathbb{P}_S$.
\end{quote}
\begin{proof}
Take $q$ in $\mathbb{P}_S$ with $\bigwedge \Phi \in q$.  
By Clause \ref{def:ConProInf}(Ind.2), $q \cup \{\phi\} \in \mathbb{P}_S$.
Assume now that $\neg\phi\in q$. Then $q \cup \{\phi\} \in \mathbb{P}_S$ would contradict
Clause \ref{def:ConProInf}(Ind.2) for $\mathbb{P}_S$. The thesis follows.
\end{proof}
By maximality of $F$ if some $q\in F$ is such that $\bigwedge\Phi\in q$, then 
$q\cup\bp{\phi}\in F$ as well, yielding that $\phi\in \Sigma_F$ as was to be shown.

\end{description}   
\end{proof}

\begin{theorem} \label{GenFilThe} 
Let $\mathrm{L}$ be a relational $\omega$-signature, $\mathcal{C}$ an infinite set of constants disjoint from $\mathrm{L}$, and $S$ be a consistency property consisting of  $\mathrm{L}(\mathcal{C})_{\infty \omega}$-sentences.
 
Assume that $F$ is a $V$-generic filter for $\mathbb{P}_S$. Then in $V[F]$ it holds that:
\begin{enumerate}
\item \label{GenFilThe-1}
The domain of 
$\mathcal{A}_F$ is exactly given by $\bp{[c]_F: c\in\mathcal{C}}$.
\item \label{GenFilThe-2}
For any $\mathrm{L}(\mathcal{C})_{\infty\omega}$-sentence $\psi$
\[
\mathcal{A}_F \vDash \psi \text{ if } \psi \in \Sigma_F.
\]
\end{enumerate}
\end{theorem}

Note the following apparently trivial corollary of the above Theorem:
\begin{corollary}\label{cor:consS}
Assume $S$ is a consistency property on $\mathrm{L}(\mathcal{C})_{\infty\omega}$  satisfying the assumptions of Thm. \ref{GenFilThe}. Then for any 
$s\in S$ $s\not\vdash\emptyset$.
%
\end{corollary}

\begin{remark}
We note that essentially the same Theorem and Corollary have been proved independently by Ben De Bondt  and Boban Velickovic (using the language of forcing via partial orders to formulate them).
\end{remark}

\begin{proof}
Assume $s\vdash\emptyset$ for some $s\in S$. Note that if $F$ is $V$-generic for $\mathbb{P}_S$ with $s\in F$, the same proof existing in $V$ of $s\vdash\emptyset$ is a proof of the same sequent in $V[F]$.
By Thm. \ref{GenFilThe} $\mathcal{A}_F\models\bigwedge s$ holds in $V[F]$.
Hence by the soundness of Tarski semantics for $\vdash$ in $V[F]$, we would get that
$\mathcal{A}_F\models\psi\wedge\neg\psi$ for some $\psi$ holds in $V[F]$. This is clearly a contradictory statement for $V[F]$.
\end{proof}

We now prove Thm. \ref{GenFilThe}:
\begin{proof}
Let (in $V[F]$) $\mathcal{A}_F $ be the structure obtained from $F$ as in Def. \ref{DefStr}. 
Since $S$ is a dense subset of $\mathbb{P}_S$, $F\cap S$ is a generic filter for $(S,\supseteq)$ as well.
 By Clause \ref{def:ConProInf}(Str.3)
\[
D_d = \{p \in S: \exists c \in \mathcal{C}, c = d \in p\}
\]
is dense in $\mathbb{P}_S$ for any $d\in\mathcal{D}$. Let $p \in F \cap D_d$. Then for some $c \in C$, $d = c \in p \subset \Sigma_F$ and $[d]_F= [c]_F$. This proves part \ref{GenFilThe-1} of the Theorem.

We now establish part \ref{GenFilThe-2}. We have to handle only the cases for $\neg$, $\bigvee$, $\exists$, $\forall$ formulae, since the atomic case and the case $\bigwedge$ can be treated exactly as we did in Fact \ref{fac:TarStrAF} and Lemma \ref{MaxFilThe}. We continue the induction as follows:


\begin{description}
\item[$\bigvee$] 
Suppose $\bigvee \Phi \in \Sigma_F$. Let $p_0 \in F$ be such that $\bigvee \Phi \in p_0$. By Clause
\ref{def:ConProInf}(Ind.4)
\[
D_{\bigvee \Phi} = \{p \in S: \exists \phi \in \Phi, \phi \in p\}
\]
is dense below $p_0$. Since $F$ is $V$-generic over $\mathbb{P}_S$ and $p_0 \in F$, there exists $p \in F \cap D_{\bigvee \Phi}$. Then for some $\phi \in \Phi$, $\phi \in p \subset \Sigma_F$ and 
\[
\mathcal{A}_F \vDash \phi,
\]
proving
\[
\mathcal{A}_F \vDash \bigvee \Phi.
\]
\item[$\exists$] 
Suppose $\exists \vec{v} \,\phi(\vec{v}) \in \Sigma_F$. Let $p_0 \in F$ such that $\exists \vec{v} \,\phi(\vec{v}) \in p_0$. 
 By Clause
\ref{def:ConProInf}(Ind.5)
\[
D_{\exists v \phi(\vec{v})} = \{p \in S: \exists \vec{c} \in \mathcal{C}^{\vec{v}}, \phi(\vec{c}) \in p\}
\]
is dense below $p_0$. Since $F$ is $V$-generic over $\mathbb{P}_S$ and $p_0 \in F$, there exists $p \in F \cap D_{\exists \vec{v} \phi(\vec{v})}$. Then for some $ \vec{c} \in \mathcal{C}^{\vec{v}}$, $\phi(\vec{c}) \in p \subset \Sigma_F$. Therefore 
\[
\mathcal{A}_F \vDash \phi(\vec{c}),
\]
hence 
\[
\mathcal{A}_F \vDash \exists \vec{v} \phi(\vec{v}).
\]

\item[$\forall$] 
Suppose $\psi = \forall\vec{x} \phi(\vec{x})$ is in $\Sigma_F'$. 
One needs to check 
\[
\mathcal{A}_F \vDash \phi(\vec{x})[\vec{x}/\vec{e}]
\]
for $\vec{e}=\ap{[e_1]_F,\dots,[e_n]_F}\in \mathcal{A}_F^{n}$.

Let $\mathcal{E}=\mathcal{C}\cup\mathcal{D}$. Then 
we have that 
\[
\mathcal{A}_F=\bp{[e]_F:\, e\in\mathcal{E}};
\]
hence
\[
\mathcal{A}_F^{<\omega}=\bp{\ap{[e_1]_F,\dots,[e_n]_F}:\, \ap{e_1,\dots,e_n}\in(\mathcal{E}^{<\omega})^{V[F]}}.
\]
A key observation is that
\[
(\mathcal{E}^{<\omega})^{V[F]}=(\mathcal{E}^{<\omega})^{V}.
\]
This gives that for any $\vec{e}\in\mathcal{A}_F^{<\omega}$
\[
\mathcal{A}_F \vDash \phi(\vec{x})[\vec{x}/\vec{e}]
\]
if and only if there are $e_1\dots e_n\in\mathcal{E}$ such that 
$\vec{e}=\ap{[e_1]_F,\dots,[e_n]_F}$ and 
\[
\mathcal{A}_F \vDash \phi(e_1,\dots,e_n).
\]
By Clause
\ref{def:ConProInf}(Ind.3), 
assuming $\forall\vec{x}\phi(\vec{x})\in\Sigma_F$, we get that
$ \phi(e_1,\dots,e_n)\in\Sigma_F$ for all $e_1,\dots,e_n\in\mathcal{E}$.
Hence in $V[F]$ it holds that
\[
\mathcal{A}_F \vDash \phi(\vec{x})[\vec{x}/\vec{e}]
\]
for all $\vec{e}\in\mathcal{A}_F^n$,
as was to be shown.

\item[$\neg$] 
Suppose $\neg \phi \in \Sigma_F$. Clause \ref{def:ConProInf}(Ind.1) ensures that $F' = [\Sigma_F \cup \{\phi \neg\}]$ is a prefilter on $\mathbb{P}_S$  containing $F$. By maximality of $F$, $\phi \neg \in F$. We know that $\phi \neg$ and $\neg \phi$ are equivalent (under any reasonable equivalence notion, for example provability, or logical consequence for Boolean valued semantics). Also the principal connective of $\phi \neg$ is of type $\bigwedge, \forall, \bigvee$ or $\exists$, for which cases the proof has already been given.
\end{description}
The above shows that for all $\mathrm{L}(\mathcal{C})_{\infty\omega}$-sentences $\psi$, if $\psi\in \Sigma_F$ then $\mathcal{A}_F\models\psi$. 
\end{proof}

\begin{remark}
One may wonder why the Theorem is proved just for consistency properties for $\mathrm{L}_{\infty\omega}$ and not for arbitrary consistency properties on $\mathrm{L}_{\infty\infty}$.
Inspecting the proof one realizes that in the case of $\forall$ we crucially used that
$\mathcal{E}^{<\omega}$ is computed the same way in $V[F]$ and in $V$. If instead we are working with
$\mathrm{L}_{\infty\lambda}$ for $\lambda>\omega$, it could be the case that $\mathcal{E}^{<\lambda}$ as computed in $V[F]$ is a strict superset of $\mathcal{E}^{<\lambda}$ as computed in $V$. In this case there is no reason to expect that 
\[
\mathcal{A}_F\models\phi(x_i:\,i<\alpha)[x_i/[e_i]_F:i<\alpha]
\]
when $\forall\vec{x}\phi(\vec{x})\in\Sigma_F$ but
$\ap{e_i:i<\alpha}\in \mathcal{E}^{<\lambda}\setminus V$.

\end{remark}

Note that it may occur that for some $\mathrm{L}(\mathcal{C})_{\infty\omega}$-sentence $\psi$, neither $\psi$ nor $\neg\psi$ belongs to any $r\in S$, hence for some $V$-generic filter $F$ for $\mathbb{P}_s$ it can be the case that $s\not\in F$ while $\mathcal{A}_F\models\psi$. For example this occurs because $S$ is a set and there are class many $\mathrm{L}(\mathcal{C})_{\infty\omega}$-sentence $\psi$.

We can prove a partial converse of the second conclusion of Thm. \ref{GenFilThe} which requires a slight strengthening of the notion of consistency property:
\begin{definition} \label{def:ConProInfMax}
Let $\mathrm{L} = \mathcal{R} \cup \mathcal{D}$, $\mathcal{D}$, $S$ be as in Def. \ref{def:ConProInf} and $\kappa$ be a cardinal greater than or equal to $|\mathcal{C}|$. 

A consistency property 
$S$  is 
\emph{$(\kappa,\lambda)$-maximal} if all its elements consist of $\mathrm{L}(\mathcal{C})_{\kappa\lambda}$-sentences and $S$ satisfies the following clause:

\begin{enumerate}
\item[(S-Max)] \label{conspropMax} For any $p\in S$
and $\mathrm{L}(\mathcal{C})_{\kappa\lambda}$-sentence $\phi$,  either $p\cup\bp{\phi}\in S$ or $p\cup\bp{\neg\phi}\in S$.
\end{enumerate}

\end{definition}

Example \ref{exm:BvalmodConsProp} (given by the finite sets of $\mathrm{L}(M)_{\kappa\lambda}$-sentences which have positive value in some fixed Boolean valued model with domain $M$) gives the standard case of a $(\kappa,\lambda)$-maximal consistency property.

\begin{proposition}
With the notation of  Thm. \ref{GenFilThe}
Assume $S$ is $(\kappa,\omega)$-maximal for some
 $\kappa\geq|\mathcal{C}|$.
Then for any $\mathrm{L}(\mathcal{C})_{\kappa\omega}$-sentence $\psi$
\[
\mathcal{A}_F \vDash \psi \text{ if and only if } \psi \in \Sigma_F.
\]
\end{proposition}
\begin{proof}
We need to prove the  ``only if'' part of the implication
 assuming $S$ is $(\kappa,\omega)$-maximal. Suppose $\psi$ is an
 $\mathrm{L}(\mathcal{C})_{\kappa\omega}$-sentence not in $\Sigma_F$.  
By $(\kappa,\omega)$-maximality of $S$ we get that
\[
D_\psi=\bp{r\in S:\psi\in r\text{ or }\neg\psi\in r}
\]
is dense in $\mathbb{P}_S$.
Since $F$ is $V$-generic for $\mathbb{P}_S$, we get that $F\cap D_\psi$ is non-empty.
Hence either $\psi\in \Sigma_F$ or $\neg\psi\in\Sigma_F$, but the first is not the case by hypothesis. Then $\neg \psi \in \Sigma_F$ and by Theorem \ref{GenFilThe} $\mathcal{A}_F\models\neg\psi$, e.g. $\mathcal{A}_F\not\models\psi$.

The desired thesis follows.
\end{proof}

Let us recall one result about $<\kappa$-cc forcing notions. Proposition \ref{StaHk2} appears in \cite{GoldsternTools}.
 
\begin{proposition} \label{StaHk2}
Let $\kappa$ be a regular cardinal and $\mathbb{P} \subset H_\kappa$ a forcing notion with the $<\kappa$-cc. Suppose $p \in \mathbb{P}$ and $\dot{\tau}$ is a $\mathbb{P}$-name such that $p \Vdash \dot{\tau} \in H_{\check{\kappa}}$, then there exists $\dot{\sigma} \in H_\kappa$ such that $p \Vdash \dot{\sigma} = \dot{\tau}$.
\end{proposition}

\begin{definition}\label{def:BmodelAS}
Given a relational $\omega$-signature $\mathrm{L}=\mathcal{R}\cup\mathcal{D}$,  an infinite set of constants $\mathcal{C}$ disjoint from $\mathrm{L}$, and  a consistency property
 $S \subset [\mathrm{L}(\mathcal{C})_{\infty \omega}]^{< \omega}$, let 
 \[
 \mathcal{A}_S=(A_S,R_S:R\in\mathcal{R},d_S: d\in\mathcal{D}\cup\mathcal{C})
 \] be 
 defined as follows:
 \begin{itemize}
 \item $A_S=\bp{\sigma \in V^{\RO(\mathbb{P}_S)} \cap H_\mu :\, 
 \Qp{\sigma\in A_{\dot{G}}}^{V^{\RO(\mathbb{P}_S)}}_{\RO(\mathbb{P}_S)}= 1_{\RO(\mathbb{P}_S)}}$, where $\mu$ is a regular cardinal big enough so that $\mathrm{L}\subseteq H_\mu$ and for any $\sigma \in V^{\RO(\mathbb{P}_S)}$ such that 
 \[
 \Qp{\sigma \in A_{\dot{G}}}^{V^{\RO(\mathbb{P}_S)}}_{\RO(\mathbb{P}_S)} = 1_{\RO(\mathbb{P}_S)},
 \] 
 one can find $\tau \in V^{\RO(\mathbb{P}_S)} \cap H_\mu$ with 
 \[
 \Qp{\tau = \sigma}^{V^{\RO(\mathbb{P}_S)}}_{\RO(\mathbb{P}_S)}=1_{\RO(\mathbb{P}_S)};
 \]
 \item $\Qp{R_S(\sigma_1,\dots,\sigma_n)}_{\RO(\mathbb{P}_S)}^{\mathcal{A}_S}= \Qp{\mathcal{A}_{\dot{G}}\models R_{\dot{G}}(\sigma_1,\dots,\sigma_n) }^{V^{\RO(\mathbb{P}_S)}}_{\RO(\mathbb{P}_S)}$ for $R\in\mathcal{R}$;
\item for $d\in\mathcal{D}\cup\mathcal{C}$,  $d_S=\check{d}$.
 \end{itemize}
 \end{definition}
 
\begin{theorem}\label{thm:mainthmAF}
Let $\mathrm{L}$ be a relational $\omega$-signature, $\mathcal{C}$ be a  set of constants disjoint from $\mathrm{L}$ of size at most $\kappa$ and $S \subset [\mathrm{L}(\mathcal{C})_{\kappa\omega}]^{< \omega}$ be a 
consistency property. Then $\mathcal{A}_S$  is a $\RO(\mathbb{P}_S)$-valued model with the mixing property, and for every $s\in S$
\[
\Qp{\bigwedge s}^{\mathcal{A}_S}_{\RO(\mathbb{P}_S)}= 
\Qp{\mathcal{A}_{\dot{G}}\models\bigwedge s}^{V^{\RO(\mathbb{P}_S)}}_{\RO(\mathbb{P}_S)}.
\]
\end{theorem}

\begin{corollary} \label{Boolean MET}
Let $\mathrm{L}$ be a relational $\omega$-signature, $\mathcal{C}$ be a  set of constants disjoint from $\mathrm{L}$ of size at most $\kappa$ and $S \subset [\mathrm{L}(\mathcal{C})_{\kappa\omega}]^{< \omega}$ be a 
consistency property.
Then for any $s \in S$ there is a $\mathsf{B}$-Boolean valued model $\mathcal{M}$ with the mixing property in which 
\[
\Qp{\bigwedge s}_\mathsf{B}^\mathcal{M} = 1_\mathsf{B}.
\]
\end{corollary}
We first prove the Corollary assuming the Theorem.
\begin{proof}
Given $s\in S$,
we let $\bool{B}=\RO(\mathbb{P}_S)\restriction \Reg{N_s}$.
Since 
\[
s\Vdash_{\mathbb{P}_S}\mathcal{A}_{\dot{G}}\models \bigwedge s,
\]
we get that $\Reg{N_s}\leq \Qp{\bigwedge s}^{\mathcal{A}_S}_{\RO(\mathbb{P}_S)}$.
In particular if we consider $\mathcal{A}_S$ as a $\bool{B}$-valued model by evaluating all atomic formulae
$R(\vec{\sigma})$ by  $\Qp{R(\vec{\sigma}}^{\mathcal{A}_S}_{\RO(\mathbb{P}_S)}\wedge \Reg{N_s}$,
we get that $ \Qp{\bigwedge s}^{\mathcal{A}_S}_{\bool{B}}=1_{\bool{B}}$. 

Note that $\bool{B}$ is not the one point Boolean algebra, since $\Reg{N_s}\neq\emptyset=0_{\RO(\mathbb{P}_S)}$ for all $s\in S$.

It is also immediate to check that $\mathcal{A}_S$ retains the mixing property also when seen as a $\bool{B}$-valued model. 

\end{proof}

We now prove Thm. \ref{thm:mainthmAF}. We need beforehand to extend the forcing relation to formulae of infinitary logic.

\begin{remark}
Given a complete Boolean algebra $\bool{B}$, an $\in$-formula 
$\phi(v_1,\ldots,v_n)$ for $\mathrm{L}_{\infty\omega}$ (for $\mathrm{L}=\bp{\in}$), and any family $\tau_1,\ldots,\tau_n \in V^\mathsf{B}$, 
$\Qp{\phi(\tau_1,\ldots,\tau_n)}^{V^{\bool{B}}}_\mathsf{B}$ denotes the $\mathsf{B}$-value of $\phi(\tau_1,\ldots,\tau_n)$ in the Boolean valued model $V^\mathsf{B}$.

The definition of $\Qp{\phi(\tau_1,\ldots,\tau_n)}^{V^{\bool{B}}}_\mathsf{B}$ is by induction on the complexity of $\phi$. It is the standard one for the atomic formulae $\Qp{\tau\in\sigma}^{V^{\bool{B}}}_\bool{B}$ and $\Qp{\tau=\sigma}^{V^{\bool{B}}}_\bool{B}$.
We extend it to all $\mathrm{L}_{\infty\omega}$ according to Def. \ref{def:boolvalsem}.
\end{remark}    
 
\begin{proof}
We first establish that $\mathcal{A}_S$ has the mixing property.
Let $\bp{\sigma_a:a\in A}$ be a family of elements of $A_S$ indexed by an antichain $A$ of
$\RO(\mathbb{P}_S)$.
Find (by the mixing property of $V^{\RO(\mathbb{P}_S)}$)
$\sigma\in V^{\RO(\mathbb{P}_S)}$ such that $\Qp{\sigma=\sigma_a}^{V^{\RO(\mathbb{P}_S)}}_{\RO(\mathbb{P}_S)}\geq a $ for all $a\in A$. By choice of $A_s$ we can suppose that $\sigma\in A_s$.
By definition of  $\mathcal{A}_S$
\[
\Qp{\sigma=\sigma_a}^{\mathcal{A}_S}_{\RO(\mathbb{P}_S)}=\Qp{\sigma=\sigma_a}^{V^{\RO(\mathbb{P}_S)}}_{\RO(\mathbb{P}_S)}\geq a 
\]
for all $a\in A$. Hence $\sigma$ is a mixing element for the family $\bp{\sigma_a:a\in A}$.

Now we prove the second part of the Theorem.
One needs to check that for any $\mathrm{L}_{\kappa \omega}$-formula $\phi(\vec{v})$ and $\sigma_1,\ldots,\sigma_n \in A_S$,
\[
\Qp{\phi(\vec{\sigma})}_{\RO(\mathbb{P}_S)}^{\mathcal{A}_S} = \Qp{\mathcal{A}_{\dot{G}} \vDash \phi(\vec{\sigma})}^{V^{\RO(\mathbb{P}_S)}}_{\RO(\mathbb{P}_S)}. 
\]
It is clear that this allows one to prove
\[
\Qp{\bigwedge s}^{\mathcal{A}_S}_{\RO(\mathbb{P}_S)}= \Qp{\mathcal{A}_{\dot{G}} \vDash\bigwedge s}^{V^{\RO(\mathbb{P}_S)}}_{\RO(\mathbb{P}_S)},
\] 
letting $\phi=\bigwedge s$.

%
%

We can prove the equality by induction on the complexity of formulae.
\begin{itemize}
\item For atomic sentences this follows by definition.
\item For $\neg$, 
\[
\Qp{\neg \phi}_{\RO(\mathbb{P}_S)}^{\mathcal{A}_S} = \neg \Qp{\phi}_{\RO(\mathbb{P}_S)}^{\mathcal{A}_S} = \neg \Qp{\mathcal{A}_{\dot{G}} \vDash \phi}_{\RO(\mathbb{P}_S)}^{V^{\RO(\mathbb{P}_S)}} = \Qp{\mathcal{A}_{\dot{G}} \not\vDash \phi}_{\RO(\mathbb{P}_S)}^{V^{\RO(\mathbb{P}_S)}} = \Qp{\mathcal{A}_{\dot{G}} \vDash \neg \phi}_{\RO(\mathbb{P}_S)}^{V^{\RO(\mathbb{P}_S)}}.
\]
\item For $\bigwedge$, 
\[
\Qp{\bigwedge \Phi}_{\RO(\mathbb{P}_S)}^{\mathcal{A}_S} = \bigwedge_{\phi \in \Phi} \Qp{\phi}_{\RO(\mathbb{P}_S)}^{\mathcal{A}_S} = \bigwedge_{\phi \in \Phi} \Qp{\mathcal{A}_{\dot{G}} \vDash \phi}_{\RO(\mathbb{P}_S)}^{V^{\RO(\mathbb{P}_S)}} = \Qp{\mathcal{A}_{\dot{G}} \vDash \bigwedge \Phi}_{\RO(\mathbb{P}_S)}^{V^{\RO(\mathbb{P}_S)}}.
\]
\item For $\exists$,

\begin{gather*}
\Qp{\exists v \phi(v,\vec{\sigma})}_{\RO(\mathbb{P}_S)}^{\mathcal{A}_S} = \bigvee_{\tau \in A_S} \Qp{\phi(\tau,\vec{\sigma})}_{\RO(\mathbb{P}_S)}^{\mathcal{A}_S} = \bigvee_{\tau \in A_S} \Qp{\mathcal{A}_{\dot{G}} \vDash \phi(\tau,\vec{\sigma})}_{\RO(\mathbb{P}_S)}^{V^{\RO(\mathbb{P}_S)}} \leq \\
\bigvee_{\tau \in V^{\RO(\mathbb{P}_S)}} \Qp{\mathcal{A}_{\dot{G}} \vDash \phi(\tau,\vec{\sigma})}_{\RO(\mathbb{P}_S)}^{V^{\RO(\mathbb{P}_S)}} = \Qp{\mathcal{A}_{\dot{G}} \vDash \exists v \phi(v,\vec{\sigma})}_{\RO(\mathbb{P}_S)}^{V^{\RO(\mathbb{P}_S)}} = \\
\Qp{\mathcal{A}_{\dot{G}} \vDash \phi(\tau_0,\vec{\sigma})}_{\RO(\mathbb{P}_S)}^{V^{\RO(\mathbb{P}_S)}} = \Qp{\phi(\tau_0,\vec{\sigma})}^{\mathcal{A}_S}_{\RO(\mathbb{P}_S)}\leq \Qp{\exists v \phi(v,\vec{\sigma})}_{\RO(\mathbb{P}_S)}^{\mathcal{A}_S}, 
\end{gather*}

where $\tau_0 \in \mathcal{A}_S$ is obtained by fullness of $V^{\RO(\mathbb{P}_S)}$ and can be supposed in $H_\mu$ by Proposition \ref{StaHk2}; while the equality in the last line holds by inductive assumptions.

\end{itemize}

\end{proof}

Let us briefly remark why genericity is needed for dealing with formulae of type $\neg$, $\bigvee$ and $\exists$ in proving the model existence Theorem. 

The case of negated formulae is dealt with by taking advantage of Def. \ref{MovNegIns}. If the negated formula is atomic, its truth value follows by the definition of $\mathcal{A}_F$. For negated formulae $\neg \phi$ with $\phi$ non-atomic, by moving a negation inside repeatedly, we find a logically equivalent formula $\psi$ where negations appear only at the atomic level of the structural tree of $\psi$; at this level there is control. In particular the operation $\phi\mapsto\phi\neg$ allows to prove Thm. \ref{thm:mainthmAF} by an induction in which one only deals with the logical symbols $\bigwedge,\forall,\bigvee$ and $\exists$. 

Genericity comes to play when dealing with formulae whose principal connective is $\bigvee$ or $\exists$. 
For both connectives the role of genericity in the proof of the corresponding inductive step is similar, so we only analyze the first one. 
The key point is that the structure $\mathcal{A}_F$ associated to a maximal filter $F$ on $\mathbb{P}_S$ is decided by which atomic formulae belong to $\Sigma_F$: any maximal consistent set of atomic formulae for $\mathrm{L}$ defines an $\mathrm{L}$-structure $\mathcal{A}_F$ by Fact \ref{fac:TarStrAF}. Now if $F$ is maximal but not $V$-generic, it may miss some $D_{\bigvee \Phi}$ for some $\Phi\in\Sigma_F$. In which case $[\Sigma_F \cup \{\phi\}]^{< \omega}$ is not a prefilter on $\mathbb{P}_S$ for any $\phi\in\Phi$, by maximality of $F$. Supposing this occurs for some $\Phi$ which is a disjunction of atomic or negated atomic formulae, we get that $\bp{\phi\neg:\phi\in \Phi}\subseteq F$, again by maximality of $F$. Hence $\mathcal{A}_F\not\vDash\bigvee\Phi$ even if $\bigvee\Phi\in F$.

\begin{remark} When working with a consistency property $S$ for $\mathrm{L}(\mathcal{C})_{\kappa \omega}$, there is a canonical way of extending it to a $(\kappa,\omega)$-maximal one. Consider the  Boolean valued model $\mathcal{A}_S$ of Def. \ref{def:BmodelAS}, let also $\bool{B}=\RO(\mathbb{P}_S)$. Then 
\[
S \subset M_S = \{t \in [\mathrm{L}(\mathcal{C} \cup \mathcal{A}_S)_{\kappa \omega}]^{<\omega} : \Qp{t}^{\mathcal{A}_S}_{\bool{B}} > 0_{\mathsf{B}}\}
\]
and $M_S$ is a $(\kappa,\omega)$-maximal consistency property for $\mathrm{L}(\mathcal{C} \cup \mathcal{A}_S)_{\kappa \omega}$. 

\end{remark}

\begin{remark} \label{Coll}
Note that in Def. \ref{def:ConProInf} the size of $\mathcal{C}$ can vary. While Thm. \ref{GenFilThe} holds for any size of $\mathcal{C}$, some sizes automatically collapse cardinals. Consider for example $\mathrm{L} = \{d_\alpha : \alpha < \omega_1\}$ and $\mathcal{C} = \{c_n : n < \omega\}$ countable. Let $S$ denote the set whose elements are the $s \in [\mathrm{L}(\mathcal{C})_{\omega_2 \omega}]^{< \omega}$ such that for some injective interpretation 
\[
c_{i_1} \mapsto \alpha_{i_1},\ldots,c_{i_n} \mapsto \alpha_{i_n}, \ \alpha_{i_j} < \omega_1,
\] 
of the constants from $\mathcal{C}$ appearing in $s$,
\[
(\omega_1,=,c_{i_k} \mapsto \alpha_{i_k},d_\alpha \mapsto \alpha) \vDash s.
\]
$S$ is readily checked to be a consistency property. Consider $\mathcal{A}_G \in V[G]$ for $G$ $V$-generic for $\mathbb{P}_S$. It is a model of $\bigwedge_{\alpha \neq \beta\in\omega_1^V} d_\alpha \neq d_\beta$, furthermore the interpretation maps 
\begin{align*}
f: \omega_1^V &\rao \bp{[d]_G:d\in\mathcal{D}} \\
\alpha &\mapsto d_\alpha^{\mathcal{A}_G}\\
&\\
g: \omega &\rao \bp{[c_n]_G: n < \omega} \\
n &\mapsto [c_n]_G
\end{align*}
are both injective. This entails that the map $\alpha\mapsto n$ if $\bp{d_\alpha=c_n}\in G$ is also injective.
Therefore $\omega_1^V$ is collapsed.
\end{remark}



\section{Mansfield's Model Existence Theorem}\label{sec:mansmodexthm}

We now prove Mansfield's Model Existence Theorem. 


\begin{theorem} \label{ManModExi} 
Let $\mathrm{L}$ be an $\omega$-signature and $S \subset P(\mathrm{L}(\mathcal{C})_{\kappa \lambda})$ a consistency property. Then for any $s \in S$ there exists a Boolean valued model $\mathcal{M}$ in which all sentences from $s$ are valid.
\end{theorem}

\begin{proof}
Fix $s_0 \in S$. Consider $\mathbb{P}_S$ the forcing notion associated to $S$, $\mathsf{B} = \RO(\mathbb{P}_S)$ the corresponding Boolean completion, $\mathbb{P}_S \upharpoonleft s_0 = \{t \in S : t \leq s_0\}$ the restriction to conditions extending $s_0$ and $\mathsf{B} \upharpoonleft s_0 = \{t \in \mathsf{B} : t \leq \Reg{N_{s_0}}\} = \RO(\mathbb{P}_S \upharpoonleft s_0)$. The Boolean valued model $\mathcal{M}$ is constructed with truth values in $\mathsf{B} \upharpoonleft s_0$ and base set the set of constants $\mathcal{C} \cup \mathcal{D}$. The interpretations of constants are given by themselves. Mimicking Mansfield's proof we keep his notation whenever possible. For any sentence $\phi$ define 
\[
L(\phi) = \bigvee \{\Reg{N_t} : \phi \in t\},
\]
and for $\phi$ atomic set 
\[
\Qp{\phi} = L(\phi).
\]
The main technical result in \cite{MansfieldConPro} is Lemma 3. Its equivalent (according to our notion of consistency property) goes as follows: 

\begin{claim} \label{Cla1} If for any $t \leq s$, $t \cup \{\phi\} \in S$, then $N_s \subseteq L(\phi)$. In particular, $\Reg{N_s} \leq L(\phi)$.
\end{claim}
\begin{proof}
Suppose  $N_s \not\subseteq L(\phi)$.
Since the family $\{N_t : t \in \mathbb{P}_S \upharpoonleft s_0\}$ is a basis of $\mathsf{B}$, basic topological facts bring that there exists some $t \in \mathbb{P}_S \upharpoonleft s_0$ such that 
\[
N_t \subseteq N_s \cap \neg L(\phi)=\Reg{\bigcup\bp{N_t: N_t\cap L(\phi)=\emptyset}}.
\]
\begin{itemize}
\item Since $N_t \subseteq N_s$ we have $t \leq s$ and the hypothesis ensures $t \cup \{\phi\} \in S$. Then, by definition of $L$, $N_{t \cup \{\phi\}} \subseteq L(\phi)$.
\item Since $N_t \subseteq \neg L(\phi)$ and $N_{t \cup \{\phi\}} \subseteq N_t$, $N_{t \cup \{\phi\}} \subseteq \neg L(\phi)$.
\end{itemize}
The two statements are incompatible. Hence $N_s \leq L(\phi)$.
\end{proof}

\begin{claim} \label{Cla2}
For any $\phi$, $L(\phi) \leq \Qp{\phi}$.
\end{claim}

\begin{proof}
We proceed by induction on the complexity of formulae. The thesis holds for atomic formulae by definition.
The other cases are dealt with as follows:

\begin{description}
\item[$\neg$] Since $S$ is closed under moving a negation inside, we only need to care about negations acting on atomic formulae. Suppose $\phi$ is atomic. We have 
\[
\Qp{\neg \phi}_{\mathsf{B}_S \upharpoonleft s_0} = \neg \Qp{\phi}_{\mathsf{B}_S \upharpoonleft s_0} = \neg L(\phi) = \bigwedge \{\neg \Reg{N_t} : \phi \in t\}.
\]
We need to prove
\[  
L(\neg \phi) = \bigvee \{\Reg{N_p} : \neg \phi \in p\} \leq \bigwedge \{\neg \Reg{N_t} : \phi \in t\}.
\]
Fix $t$ containing $\phi$. For any $p \ni \neg \phi$, $p$ and $t$ are incompatible by Clause \ref{conspropCon}(Con). Remark \ref{rema1} ensures $\Reg{N_p} \leq \neg \Reg{N_t}$. Then 
\[
\bigvee_{p \ni \neg \phi} \Reg{N_p} \leq \neg \Reg{N_t}.
\]
Since this is true for any $t \ni \phi$, 
\[
\bigvee_{p \ni \neg \phi} \Reg{N_p} \leq \bigwedge_{t \ni \phi} \neg \Reg{N_t}.
\]

\item[$\bigwedge$] 
Suppose by induction that the result holds for any $\phi \in \Phi$. Let $\bigwedge \Phi \in s$. Then for any $t$ extending $s$ and any $\phi \in \Phi$, $t \cup \{\phi\} \in S$. By Claim \ref{Cla1}, $\Reg{N_s} \leq L(\phi)$ for any $\phi \in \Phi$. By the induction hypothesis $\Reg{N_s} \leq L(\phi) \leq \Qp{\phi}$. As this holds for any $\phi \in \Phi$, $\Reg{N_s} \leq \bigwedge_{\phi \in \Phi} \Qp{\phi} = \Qp{\bigwedge \Phi}$. This holds for any $s$ such that $\bigwedge \Phi \in s$, hence
\[
L(\bigwedge \Phi) = \bigvee \{\Reg{N_s} : \bigwedge \Phi \in s\} \leq \Qp{\bigwedge \Phi}.
\]

\item[$\bigvee$] 
Suppose the result true for any $\phi \in \Phi$. If $L(\bigvee \Phi) \nleq \Qp{\bigvee \Phi}$, $L(\bigvee \Phi) \wedge \neg \Qp{\bigvee \Phi} \neq \emptyset$. Therefore there exists some $t \leq s_0$ such that $N_t \subseteq L(\bigvee \Phi)$ and $N_t \subseteq \neg \Qp{\bigvee \Phi}$.
\begin{itemize}
\item By the first inclusion, since $N_t$ is open and $\bigcup \{\Reg{N_p} : \bigvee \Phi \in p\}$ is dense in $\bigvee \{\Reg{N_p} : \bigvee \Phi \in p\}$, there exists some $p'$ containing $\bigvee \Phi$ such that $N_t \cap \Reg{N_{p'}}$ is non-empty. Hence we can find $p \leq t,p'$. Since $\bigvee \Phi \in p'$, $\bigvee \Phi \in p$.
\item By the second inclusion (and $p \leq t$), 
\[
N_p \subseteq N_t \subseteq \neg \bigvee_{\phi \in \Phi} \Qp{\phi} = \bigwedge_{\phi \in \Phi} \neg \Qp{\phi}.
\]
\end{itemize} 
 
(Ind.4) ensures that for some $\phi_0 \in \Phi$, $q = p \cup \{\phi_0\} \in S$. By induction hypothesis, $L(\phi_0) \leq \Qp{\phi_0}$, hence $\neg \Qp{\phi_0} \leq \neg L(\phi_0)$. Therefore 
\[
\hspace{1,3cm} N_q \subseteq N_p \subseteq \bigwedge_{\phi \in \Phi} \neg \Qp{\phi} \leq \neg \Qp{\phi_0} \leq \neg L(\phi_0) = \bigwedge \{\neg \Reg{N_t} : \phi_0 \in t\} \subseteq \neg \Reg{N_q},
\]
a contradiction.

\item[$\forall$]
Let $s \in S$ contain $\forall \vec{v} \phi(\vec{v})$. Then for any $t$ extending $s$ and any $\vec{c} \in (C \cup D)^{\vec{v}}$, $t \cup \{\phi(\vec{c})\} \in S$. By Claim \ref{Cla1} and the induction hypothesis, 
\[
\Reg{N_s} \leq L(\phi(\vec{c})) \leq \Qp{(\phi(\vec{c}))}
\]
for any $\vec{c} \in (C \cup D)^{\vec{v}}$. Then $\Reg{N_s} \leq \bigwedge_{\vec{c} \in (C \cup D)^{\vec{v}}} \Qp{\phi(\vec{c})} = \Qp{\forall \vec{v} \phi(\vec{v})}$. All this was done for any $s$ such that $\forall \vec{v} \phi(\vec{v}) \in s$. Then we may take the sup over such sets to obtain 
\[
L(\forall \vec{v} \phi(\vec{v})) = \bigvee \{\Reg{N_s} : \forall \vec{v} \phi(\vec{v}) \in s\} \leq \Qp{\forall \vec{v} \phi(\vec{v})}.
\] 

\item[$\exists$] 
Suppose the result true for any $\phi(\vec{c})$. If $L(\exists \vec{v} \phi(\vec{v})) \nleq \Qp{\exists \vec{v} \phi(\vec{v})}$, $L(\exists \vec{v} \phi(\vec{v})) \wedge \neg \Qp{\exists \vec{v} \phi(\vec{v})} \neq \emptyset$. Then there exists some $t \leq s_0$ such that $N_t \subseteq L(\exists \vec{v} \phi(\vec{v}))$ and $N_t \subseteq \neg \Qp{\exists \vec{v} \phi(\vec{v})}$.
\begin{itemize}
\item By the first inclusion, since $N_t$ is open and $\bigcup \{\Reg{N_p} : \exists \vec{v} \phi(\overline{v}) \in p\}$ is dense in $\bigvee \{\Reg{N_p} : \exists \vec{v} \phi(\vec{v}) \in p\}$, there exists some $p'$ containing $\exists \vec{v} \phi(\vec{v})$ such that $N_t \cap \Reg{N_{p'}}$ is non-empty. Hence we can find $p \leq t,p'$ with $\exists \vec{v} \phi(\vec{v}) \in p$.
\item By the second inclusion, 
\[
N_p \subseteq N_t \subseteq \neg \bigvee_{\vec{c} \in (\mathcal{C} \cup \mathcal{D})^{\vec{v}}} \Qp{\phi(\vec{c})} = \bigwedge_{\vec{c} \in (\mathcal{C} \cup \mathcal{D})^{\vec{v}}} \neg \Qp{\phi(\vec{c})}.
\]
\end{itemize} 
 
(Ind.5) ensures that for some $\vec{c}_0 \in (\mathcal{C} \cup \mathcal{D})^{\vec{v}}$, $q = p \cup \{\phi(\overline{c}_0)\} \in S$. By the induction hypothesis, $L(\phi(\vec{c}_0)) \leq \Qp{\phi(\vec{c}_0)}_\mathsf{B}$, hence $\neg \Qp{\phi(\vec{c}_0)} \leq \neg L(\phi(\vec{c}_0))$. Therefore
\[
N_q \subseteq N_p \subseteq \bigwedge_{\overline{c} \subseteq \mathcal{C}} \neg \Qp{\phi(\overline{c})}_\mathsf{B} \subseteq \neg \Qp{\phi(\overline{c}_0)}_\mathsf{B} \subseteq \neg L(\phi(\overline{c}_0)) = \bigwedge \{\neg \Reg{N_t} : \phi(\overline{c}_0) \in t\} \subseteq \neg \Reg{N_q},
\]
a contradiction.
\end{description}

\end{proof}

Now we can check that $\mathcal{M}$ is a Boolean valued model. 

\begin{itemize}
\item Since for any $t$ in $S$ and any $c \in \mathcal{C} \cup \mathcal{D}$, $t \cup \{c = c\} \in S$, $L(c = c) = 1_{\mathsf{B}_S \upharpoonleft s_0}$.
\item Let $c = d \in s$. Then for any $t$ extending $s$, $t \cup \{d = c\} \in S$, hence $\Reg{N_s} \leq L(d = c)$. Since the previous holds for any $s$ containing $c = d$,
\[
\bigvee \{\Reg{N_s} : c = d \in s\} \leq L(d = c).
\]
Since $c = d$ and $d = c$ are atomic, 
\[
\Qp{c = d} = L(c = d) = \bigvee \{\Reg{N_s} : c = d \in s\} \leq L(d = c) = \Qp{d  = c}.
\]
\item Let $c_1 = d_1, \ldots, c_n = d_n, \phi(c_1,\ldots,c_n) \in s$ with $\phi$ atomic. Then for any $t$ extending $s$, $t \cup \{\phi(d_1,\ldots,d_n)\} \in S$, hence $\Reg{N_s} \leq L(\phi(d_1,\ldots,d_n))$. Since the previous holds for any $s$ containing $c_1 = d_1, \ldots, c_n = d_n, \phi(c_1,\ldots,c_n)$,
\[
\bigvee \{\Reg{N_s} : c_1 = d_1, \ldots, c_n = d_n, \phi(c_1,\ldots,c_n) \in s\} \leq L(\phi(d_1,\ldots,d_n)).
\]
Since $\phi$ and $c_i = d_i$ are atomic, 
\begin{align*}
\Qp{c_1 = d_1} \wedge \ldots \wedge & \Qp{c_n = d_n} \wedge \Qp{\phi(c_1,\ldots,c_n)} = \\
L(c_1 = d_1) \wedge \ldots \wedge & \ L(c_n = d_n) \wedge L(\phi(c_1,\ldots,c_n)) = \\
\bigvee \{\Reg{N_s} : c_1 = d_1, \ldots, c_n = d_n, \ &\phi(c_1,\ldots,c_n) \in s\} \leq L(\phi(d_1,\ldots,d_n)) = \\
& \Qp{\phi(d_1,\ldots,d_n)}.
\end{align*}
To prove the equality between lines two and three it is enough to check $L(\phi) \wedge L(\psi) = \bigvee \{\Reg{N_s} : \phi,\psi \in s\}$. By definition of $L$,
\begin{align*}
L(\phi) \wedge L(\psi) & = \\
\bigvee \{\Reg{N_s} : \phi \in s\} \wedge \bigvee \{\Reg{N_t} : \psi \in t\} & = \bigvee \{\bigvee \{\Reg{N_s} : \phi \in s\} \wedge \Reg{N_t} : \psi \in t\} = \\
\bigvee \{\bigvee \{\Reg{N_s} \wedge \Reg{N_t} : \phi \in s\} : \psi \in t\} & = \bigvee \{\Reg{N_s} \wedge \Reg{N_t} : \phi \in s \wedge \psi \in t\} = \\
\bigvee \{\Reg{N_q} :\  & \phi,\psi \in q\}.
\end{align*}
\end{itemize}

It remains to conclude that $s_0$ is valid in $\mathcal{M}$. Now for any $\phi \in s_0$ and $t \leq s_0$, $\phi \in t$, hence Claim \ref{Cla2} ensures $1_{\mathsf{B}_S \upharpoonleft s_0} = \Reg{N_{s_0}} \leq L(\phi) \leq \Qp{\phi}_\mathsf{B}$ for any such $\phi$.
\end{proof}

\begin{remark}
We note that a key assumption for Mansfield's result is that $\mathrm{L}$ is a first order signature.
This is crucially used in the proof that \ref{eqn:subslambda} holds for the relation symbols of $\mathrm{L}$ in the structure $\mathcal{M}$: since these relation symbols are finitary, in the proof above we just had to distribute finitely many infinite disjunctions; this distributive law holds for any complete Boolean algebra. 
If we dealt with a relational $\omega_1$-signature we might have had to distribute countably many infinitary disjunctions to establish \ref{eqn:subslambda}; this is possible only under very special circumstances on $\mathbb{P}_S$.
 
 We do not know whether this result can be established for arbitrary $\lambda$-signatures. We conjecture this is not the case.
\end{remark}

\begin{remark}
The model produced by Mansfield's theorem does not verify the mixing property in general. Consider again the setting in remark \ref{Coll}.\vspace{0,2cm}

Fix $\alpha < \omega_1$. For each $n < \omega$ define

\[
\phi_n : \hspace{0,3cm} c_n = d_\alpha \ \wedge \bigwedge_{m < n} c_m \neq d_\alpha.
\]

\noindent Then $\{\Reg{N_{\{\phi_n\}}} : n < \omega\}$ is an antichain. Assign $c_{n-1}$ to each $\Reg{N_{\{\phi_n\}}}$. We show that for no element $m$ in the structure provided by Mansfield theorem with respect to the collapsing consistency property we have

\[
\Reg{N_{\{\phi_n\}}} \leq \Qp{m = c_{n-1}} = \bigvee \{\Reg{N_s} : m = c_{n-1} \in s\}.
\]
\vspace{0cm}

There are three possibilities for $m$.

\begin{itemize}
\item $m = c_{n_0}$: Note that because the interpretations generating the consistency property are injective the set

\[
\{t \in S : \bigwedge_{n \neq m} c_n \neq c_m \in t\} 
\]
\vspace{0cm}

\noindent is dense. Then $\Qp{c_{n_0} = c_{n_0-1}} = 0$ and it cannot be that 

\[
\Reg{N_{\{\phi_{n_0}\}}} \leq \Qp{c_{n_0} = c_{n_0-1}}.
\]
\vspace{0cm}

\item $m = d_\alpha$: Take any $n < \omega$. The set $\{t \in S : c_{n-1} \neq d_\alpha \in t\}$ is dense below $\{\phi_n\}$. Then we cannot have $\Reg{N_{\{\phi_{n}\}}} \leq \bigvee \{\Reg{N_t} : d_\alpha = c_{n-1} \in t\}$.\medskip

\item $m = d_\beta$, $\beta \neq \alpha$: Take any $n < \omega$. Because the constant $d_\beta$ does not appear in $\{\phi_n\}$ and the sentence $\phi_n$ only forces $c_{n-1}$ not be $d_\alpha$, we can suppose that the interpretation that generates $\{\phi_n\}$ is such that $c_{n-1}$ is interpreted differently from $d_\beta$, proving $\{\phi_n, c_{n-1} \neq d_\beta\} \in S$. Then we cannot have $\Reg{N_{\{\phi_{n}\}}} \leq \bigvee \{\Reg{N_t} : d_\beta = c_{n-1} \}$ since $\Reg{N_{\{\phi_{n}, c_{n-1} \neq d_\beta\}}} \leq \Reg{N_{\{\phi_{n}\}}}$.
 
\end{itemize}

\end{remark}


\section{Proofs of model theoretic results}\label{sec:proofmodthres}

In this section we prove that $\mathrm{L}_{\infty \omega}$ with Boolean valued semantics has a completeness theorem, the Craig interpolation property and also an omitting types theorem. The last two results
generalize to $\mathrm{L}_{\infty \omega}$ results obtained in \cite{KeislerInfLog} for $\mathrm{L}_{\omega_1 \omega}$ by replacing Tarski semantics with Boolean valued semantics. We also provide the missing details for the general $\mathrm{L}_{\infty \infty}$ results.

\subsection{Proof of Thm. \ref{them:boolcompl}}
\begin{proof}
\ref{thm:boolcomp3} implies \ref{thm:boolcomp2} and \ref{thm:boolcomp2} implies 
\ref{thm:boolcomp1} are either standard or trivial.

Assume \ref{thm:boolcomp3} fails, we show that \ref{thm:boolcomp1} fails as well. Assume $T\not\vdash S$ with $T,S$ sets of $\mathrm{L}_{\infty\omega}$-formulae. Let $\mathcal{C}$ be an infinite set of fresh constants and
let $R$ be the family of finite sets $r\subseteq \mathrm{L}(\mathcal{C})$ such that

\begin{itemize}
\item
$r\cup T\not\vdash S$,
\item
any $\phi\in r$ contains only finitely many constants from 
$\mathcal{C}$.
\end{itemize}

Provided $R$ is a consistency property, this gives that $\mathcal{A}_R$ witnesses that
$T\not\models_{\mathrm{Sh}}S$ as:
\begin{itemize}
\item 
$\Qp{\psi}^{\mathcal{A}_R}=1_{\RO(\mathbb{P}_R)}$ for all $\psi\in T$, 
since for any $\psi\in T$
\[
E_\psi=\bp{r\in R:\, \psi\in r}
\]
is dense in $\mathbb{P}_R$;
\item
$\Qp{\phi}^{\mathcal{A}_R}=0_{\RO(\mathbb{P}_R)}$ for all $\phi\in S$, since for any such $\phi$
\[
F_\phi=\bp{r\in R:\, \neg\phi\in r}
\]
is dense in $\mathbb{P}_R$: note that $r\cup\bp{\neg\phi}\cup T\vdash S$ if and only if 
$r\cup T\vdash S\cup\bp{\phi}$, which -if $\phi\in S$- amounts to say that $r\not\in R$.
\end{itemize}

Now we show that $R$ is a consistency property:
\begin{itemize}
\item[(Con)] Trivial by definition of $R$, since the calculus is sound.

\item[(Ind.1)] Trivial since for any $\neg\phi$ in $r$, $\bigwedge r\vdash \bigwedge (r\cup\bp{\phi\neg})$ and conversely.

\item[(Ind.2)] Trivial since $r\vdash\bigwedge( r\cup\bp{\phi})$ and conversely if 
$\bigwedge\Phi\in r$ and $\phi\in \Phi$.

\item[(Ind.3)] Trivial since $r\vdash\bigwedge( r\cup\bp{\phi(c)})$ and conversely if 
$\forall v\,\phi(v)\in r$.

\item[(Ind.4)] Let $\bigvee \Sigma \in r\in R$. Since $r\in R$, $r\cup T\not\vdash S$. By contradiction suppose that for all $\sigma\in\Sigma$, $r\cup\bp{\sigma}\cup T\vdash S$. Then, by the left $\bigvee$-rule of the calculus $r\cup\bp{\bigvee\Sigma}\cup T\vdash S$. This contradicts $r\in R$, since $r=r\cup\bp{\bigvee\Sigma}$.

\item[(Ind.5)] Suppose $\exists v\, \varphi(v) \in r$. Pick $c\in\mathcal{C}$ which does not appear in any formula in $r$. It exists by definition of $R$.
Suppose $r\cup\bp{\varphi(c)}\cup T\vdash S$. Since $c$ does not appear in any formula of $r\cup S$,
$r\cup\bp{\exists x\,\varphi(x)}\vdash S$ (applying the rules of the calculus). This contradicts $r\in R$, since
$r=r\cup\bp{\exists x\,\varphi(x)}$.

\item[(Str.1,2,3)] All three cases follow from 
standard applications of the rules of the calculus for equality.
\end{itemize}
\end{proof}

\subsection{Proof of Thm. \ref{thm:craigint}}
\begin{proof}
Fix a set $\mathcal{C}$ of fresh constants for $\mathrm{L}$ of size $\kappa$.
Consider $X_\phi$ the set of all $\mathrm{L}(\mathcal{C})_{\kappa \omega}$-sentences $\chi$ such that:
\begin{itemize}
\item all non logical symbols from $\mathrm{L}$ appearing in $\chi$ also appear in $\phi$,
\item only a finite number of constants from $\mathcal{C}$ are in $\chi$.
\end{itemize} 
Define $X_\psi$ similarly. Consider $S$ the set of finite sets of $\mathrm{L}(\mathcal{C})_{\kappa \omega}$-sentences $s$ such that: 
\begin{itemize}
\item $s = s_1 \cup s_2$,
\item $s_1 \subset X_\phi$,
\item $s_2 \subset X_\psi$,
\item if $\theta,\sigma \in X_\phi \cap X_\psi$ are such that
\begin{itemize}
\item no constant symbols of $\mathcal{C}$ appears in either $\theta$ or $\sigma$,
\item $\vDash_{\mathrm{BVM}} \bigwedge s_1 \rightarrow \theta$ and $\vDash_{\mathrm{BVM}} \bigwedge s_2 \rightarrow \sigma$,
\end{itemize} 
then $\theta \wedge \sigma$ is Boolean consistent.
\end{itemize}

We will later show that $S$ is a consistency property. Assuming this fact as granted, we now show why this provides the interpolant. The Model Existence Theorem \ref{thm:mainthmAF} grants that any $s \in S$ has a Boolean valued model. By hypothesis $\vDash_{\mathrm{BVM}} \phi \rightarrow \psi$, thus the set $\{\phi, \neg \psi\}$ is not consistent and it cannot belong to $S$. 

Now we search what property the set $\{\phi, \neg \psi\}$ misses. They have no constant from $\mathcal{C}$ since they are $\mathrm{L}$-sentences. The sets $s_1$ and $s_2$ are given by $\{\phi\}$ and $\{\neg \psi\}$. So, the last property must fail. This means that there exist $\theta,\sigma \in X_\phi \cap X_\psi$ with no constant symbols of $\mathcal{C}$ in either of them and such that $\vDash_{\mathrm{B}} \phi \rightarrow \theta$, $\vDash_{\mathrm{BVM}} \neg \psi \rightarrow \sigma$ and $\theta \wedge \sigma$ is not consistent. The last assertion gives
\[
\vDash_{\mathrm{BVM}} \theta \rightarrow \neg \sigma.
\]
This together with  
\[
\vDash_{\mathrm{BVM}} \neg \sigma \rightarrow \psi
\]
implies 
\[
\vDash_{\mathrm{BVM}} \theta \rightarrow \psi.
\]
Recall that $\theta,\sigma$ have no constant symbol from $\mathcal{C}$, hence the interpolant is given by the $\mathrm{L}_{\kappa \omega}$-sentence $\theta$.
\medskip

It remains to check that $S$ is a consistency property. 

\begin{itemize}
\item[(Con)] 
The very definition of $S$ then gives that if some $s\in S$ is such that  $\theta,\neg\theta\in S$, then
$\theta,\neg\theta\in s_1\subseteq X_\phi$ or $\theta,\neg\theta\in s_2\subseteq X_\psi$.
Towards a contradiction 
w.l.o.g. we can suppose that for some $s=s_1\cup s_2\in S$ and $\theta\in X_\phi$, 
$\theta,\neg \theta \in s_1$. Consider any sentence $\chi'\in X_{\phi}\cap X_\psi$ such that $\vDash_\mathrm{BVM} \bigwedge s_2 \rightarrow \chi'$. Because $s_1$ is contradictory we have $\vDash_\mathrm{BVM} \bigwedge s_1 \rightarrow \neg \chi'$. But $\chi' \wedge \neg \chi'$ is not Boolean consistent, a contradiction.
\item[(Ind.1)] Suppose $\neg \chi \in s_1 \subseteq s$. Because $s_1 \cup \{\chi \neg\}$ and $s_1$ are equivalent, any sentence $\chi'$ such that $\vDash_\mathrm{BVM} \bigwedge s_1 \cup \{\chi \neg\} \rightarrow \chi'$ also verifies $\vDash_\mathrm{BVM} \bigwedge s_1 \rightarrow \chi'$. Then, $s \cup \{\chi \neg\} \in S$.
\item[(Ind.2)] Suppose $\chi \in \Phi$ and $\bigwedge \Phi \in s_1 \subseteq s$. Because $\bigwedge s_1$ and $\bigwedge s_1 \cup \{\chi\}$ are equivalent, $s \cup \{\chi\} \in S$. 
\item[(Ind.3)] Suppose $\forall v \chi(v) \in s_1 \subseteq s$ and $c \in \mathcal{C} \cup \mathcal{D}$. Because $\bigwedge s_1$ and $\bigwedge s_1 \cup \{\chi(c)\}$ are equivalent, $s \cup \{\chi(c)\} \in S$.
\item[(Ind.4)] Let $\bigvee \Sigma \in s_1 \subseteq s$. By contradiction we suppose that for no $\sigma \in \Sigma$, $s \cup \{\sigma\} \in S$. This means that for each $\sigma \in \Sigma$ there exist $\chi_\sigma^1, \chi_\sigma^2 \in X_\phi \cap X_\psi$ such that 
\[
\vDash_\mathrm{BVM} \bigwedge (s_1 \cup \{\sigma\}) \rightarrow \chi_\sigma^1  \emph{ and } \vDash_\mathrm{BVM} \bigwedge s_2 \rightarrow \chi_\sigma^2,
\]
but $\chi_\sigma^1 \wedge \chi_\sigma^2$ is inconsistent. 
Then
\begin{align*}
&\vDash_\mathrm{BVM} \bigwedge (s_1 \cup \{\bigvee \Sigma\}) \rightarrow \bigvee \{\chi_\sigma^1 : \sigma \in \Sigma\}  \emph{ and } \\ 
&\vDash_\mathrm{BVM} \bigwedge s_2 \rightarrow \bigwedge \{\chi_\sigma^2 : \sigma \in \Sigma\}.
\end{align*}
Note that $s_1 \cup \bp{\bigvee \Sigma} = s_1$. Because $\chi_\sigma^1 \wedge \chi_\sigma^2$ is Boolean inconsistent for each $\sigma \in \Sigma$, so is 
\begin{gather*}
\bigvee \{\chi_\sigma^1 : \sigma \in \Sigma\} \wedge \bigwedge \{\chi_\sigma^2 : \sigma \in \Sigma\} \equiv_{\mathrm{BVM}} \\
\bigvee \{\chi_{\sigma'}^1 \wedge \bigwedge \{\chi_\sigma^2 : \sigma \in \Sigma\}: \sigma' \in \Sigma\} \equiv_{\mathrm{BVM}} \\
\bigvee \bigwedge \{\chi_{\sigma'}^1 \wedge \chi_\sigma^2 : \sigma \in \Sigma \wedge \sigma' \in \Sigma\} \models_{\mathrm{BVM}} \\
\bigvee \bigwedge \{\chi_{\sigma}^1 \wedge \chi_\sigma^2 : \sigma \in \Sigma\},
\end{gather*}
since the latter is Boolean inconsistent.

Then $\theta$ being $ \bigvee \{\chi_\sigma^1 : \sigma \in \Sigma\}$ and $\sigma$ being  
$\bigwedge \{\chi_\sigma^2 : \sigma \in \Sigma\}$ witness that $s = s_1 \cup s_2 \not\in S$.

\item[(Ind.5)] Suppose $\exists v \chi(v) \in s_1 \subseteq s$ and consider $c \in \mathcal{C}$ a constant not appearing in $s$, which exists by the clause on the number of constants from $\mathcal{C}$ in sentences in $X_\phi$. Let us check $s \cup \{\chi(c)\} \in S$. For this take $\theta,\sigma \in X_\phi \cap X_\psi$ such that $\vDash_{\mathrm{Sh}} \bigwedge s_1 \cup \{\chi(c)\} \rightarrow \theta$ and $\vDash_{\mathrm{Sh}} \bigwedge s_2 \rightarrow \sigma$ with no constants from $\mathcal{C}$ either in $\theta$ or in $\sigma$. We must show that $\theta\wedge \sigma$ is Boolean satisfiable. It is enough to prove $\vDash_{\mathrm{Sh}} s_1 \rightarrow \theta$. Consider $\mathcal{M}$ a Boolean valued model for $\mathrm{L}\cup\bp{c}$ with the mixing property such that $\mathcal{M} \vDash s_1$. Since $\exists v \chi(v) \in s_1$, $\Qp{\exists v \chi(v)}_\mathsf{B}^\mathcal{M} = 1_\mathsf{B}$;
since $\mathcal{M}$ is full, we can find $\tau\in  M$ such that 
\[
\Qp{\exists v \chi(v)}_\mathsf{B}^\mathcal{M} = \Qp{\chi(\tau)}_\mathsf{B}^\mathcal{M} = 1_\mathsf{B}. 
\]
Consider $\mathcal{M}'$ to be the model obtained from $\mathcal{M}$ reinterpreting all symbols of $\mathrm{L}$ the same way, but mapping now $c$ to $\tau$.

Then $\mathcal{M}'\models\bigwedge s_1\cup\bp{\phi(c)}$, hence $\Qp{\theta}^{\mathcal{M}'}_{\bool{B}}=1_{\bool{B}}$ as well.
Since $c$ does not appear in $\theta$ we get that 
 $\Qp{\theta}^{\mathcal{M}}=\Qp{\theta}^{\mathcal{M}'}=1_{\bool{B}}$.
 
\item[(Str.1,2,3)] All three cases follow from $\bigwedge s_1$ and $\bigwedge s_1 \cup \{\chi\}$ being $\mathrm{BVM}$-equivalent when $\chi$ is the relevant formula of each clause.
\end{itemize} 
\end{proof}

\subsection{Proof of Thm. \ref{thm:omittypthm}}

\begin{proof}
Fix a set $\mathcal{C}=\bp{c_i: i<\kappa}$ of constants. Consider $\mathrm{L}(\mathcal{C})_{T,\mathcal{F}}$ the set of all sentences obtained by replacing in the $\mathrm{L}_{T,\mathcal{F}}$-formulae with free variables in $\bp{v_i:i\in\omega}$ all occurrences of these finitely many free variables by constants from $\mathcal{C}$. The consistency property $S$ has as elements the sets
\begin{align*}
s = s_0 \cup \bp{\bigvee \bp{\phi[c_{\sigma_\Phi(0)}, \ldots, c_{\sigma_\Phi(n_\Phi-1)}]: \phi \in \Phi}: \, \Phi \in \mathcal{F}_0},
\end{align*}
where: 
\begin{itemize}
    \item $s_0$ is a finite set of $\mathrm{L}(\mathcal{C})_{T,\mathcal{F}}$ sentences,
    \item only finitely many constants from $\mathcal{C}$ appear in $s_0$, 
    \item $\mathcal{F}_0$ is a finite subset of $\mathcal{F}$, 
    \item $\sigma_\Phi:\omega\to \mathcal{C}$ for all $\Phi\in \mathcal{F}_0$, and
    \item $T \cup s_0$ has a Boolean valued model.
\end{itemize} 

We first check that $S$ is a consistency property: consider $s \in S$ and $\psi \in s$. First of all, by definition of $S$ and the Completeness  Thm. \ref{them:boolcompl} we can fix a mixing model $\mathcal{M}$ of $s_0 \cup T$. We deal with two cases. If $\psi \in s_0 \cup T$, then $\mathcal{M} \vDash \psi$ allows to find the correspondent formula (here one also uses that only finitely many constants from $\mathcal{C}$ occur in $s_0$). Thus we only need to deal with the case 
\[
\psi = \bigvee \bp{\phi[c_{\sigma_\Phi(0)}, \ldots, c_{\sigma_\Phi(n_\Phi-1)}]: \phi \in \Phi}
\]
for some $\Phi \in \mathcal{F}$ and $\sigma_\Phi:\omega\to\mathcal{C}$. We need to find some $\phi \in \Phi$ such that $s \cup \{\phi\} \in s$. Denote $d_0,\ldots,d_m \in \mathcal{C}$ the constants in $s_0$ from $\mathcal{C}$ that are not $c_{\sigma(0)}, \ldots, c_{\sigma(n_\Phi-1)}$ and write $s_0$ as
\[
s_0[c_{\sigma_\Phi(0)}, \ldots, c_{\sigma_\Phi(n_\Phi-1)},d_0,\ldots,d_m]
\]
with its constant symbols displayed.
Since 
\[
\mathcal{M} \vDash T \cup s_0,
\]
we have
\[
\mathcal{M} \vDash \exists v_0 \ldots v_{n_\Phi-1} \exists w_0 \ldots w_m \bigwedge s_0[v_0,\ldots,v_{n_\Phi-1},w_0,\ldots,w_m].
\]
By the Theorem assumptions, since 
\[
\exists v_0 \ldots v_{n_\Phi-1} \exists w_0 \ldots w_m \bigwedge s_0[v_0,\ldots,v_{n_\Phi-1},w_0,\ldots,w_m]
\]
is an $\mathrm{L}_{T,\mathcal{F}}$-formula, we get that for some $\phi \in \Phi$,
\[
T \cup \bp{\exists v_0 \ldots v_{n_\Phi-1} \exists w_0 \ldots w_m \bigwedge s_0[v_0,\ldots,v_{n_\Phi-1},w_0,\ldots,w_m] \wedge \phi[v_0,\ldots,v_{n_\Phi-1}] }
\]
has an $\mathrm{L}_{T,\mathcal{F}}$-model $\mathcal{N}$, which again by completeness can be supposed to be mixing. Make $\mathcal{N}$ an $\mathrm{L}(\mathcal{C})_{T,\mathcal{F}}$-structure by choosing an interpretation of the constants from $\mathcal{C}$ such that $c_{\sigma_\Phi(0)}, \ldots, c_{\sigma_\Phi(n_\Phi-1)}$ are  assigned to $v_0,\ldots,v_{n_\Phi-1}$ and $d_0,\ldots,d_m$ are assigned to $w_0,\ldots,w_m$. Then 
\[
s_0 \cup \{\phi\} \cup \bp{\bigvee \bp{\phi[c_{\sigma_\Phi(0)}, \ldots, c_{\sigma_\Phi(n_\Phi-1)}]: \phi \in \Phi}: \, \Phi \in \mathcal{F}_0} \in S.
\]
This concludes the proof that $S$ is a consistency property.

It is now straightforward to check that for all $\Phi\in\mathcal{F}$ and $\sigma:\omega\to\mathcal{C}$
\[
D_{\Phi,\sigma}=\bp{s\in S: \bigvee\bp{\phi[c_{\sigma(0)}, \ldots, c_{\sigma(n_\Phi-1)}]:\phi\in \Phi}\in s}
\]
is dense in $\mathbb{P}_S$ and that for all $\phi\in T$ so is $D_\phi=\bp{s\in S:\phi\in S}$.

By the Model Existence Theorem there is a model $\mathcal{M}$ of 
\[
T \cup \bp{\bigvee \bp{\phi[c_{\sigma(0)}, \ldots, c_{\sigma(n_\Phi-1)}]: \phi \in \Phi}: \, \Phi \in \mathcal{F}, \sigma:\omega\to \mathcal{C}}
\]
in which all the elements are the interpretation of some constant from $\mathcal{C}$. Thus $\mathcal{M}$  models the theory 
\[
T \cup \bp{\bigwedge_{\Phi \in \mathcal{F}} \forall v_0 \ldots v_{n_\Phi-1} \bigvee \Phi(v_0,\ldots,v_{n_\Phi-1}},
\]
as required.

\end{proof}

\subsection{Proof of Thm. \ref{thm:craigint2}}

\begin{proof}
The proof is a small twist of the proof of Thm. \ref{thm:craigint} with two differences. First, when obtaining a Boolean valued model for an element of the consistency property one needs to apply Thm. \ref{ManModExi} instead of \ref{GenFilThe}. Secondly, when proving in the proof of Thm. \ref{thm:craigint} that $S$ is a consistency property, the existential case strongly uses the Boolean valued models being full, thus also this part of that proof needs a revision.

\begin{itemize}
\item[(Ind.5)] Suppose $\exists \vec{v} \varphi(\vec{v}) \in s_1 \subseteq s$ and consider $\vec{c} \in \mathcal{C}^{\vec{v}}$ a sequence of constants not appearing in $s$, which exists by the clause on the number of constants from $\mathcal{C}$ in sentences\footnote{Note that $\vec{v}$ can be an infinite string of variables of length less than $\lambda$, nonetheless in $\varphi(\vec{v})$ only finitely many constants from $\mathcal{C}$ appears in it as well as in any other formula of $s$.} in $X_\phi$. Let us check $s \cup \{\varphi(\vec{c})\} \in S$. For this take $\theta,\sigma \in X_\phi \cap X_\psi$ such that $\vDash_{\mathrm{BVM}} \bigwedge s_1 \cup \{\varphi(\vec{c})\} \rao \theta$ and $\vDash_{\mathrm{BVM}} \bigwedge s_2 \rao \sigma$. We must show that $\theta \wedge \sigma$ is Boolean consistent. It is enough to prove $\vDash_{\mathrm{BVM}} s_1 \rao \theta$. Consider $\mathcal{M}$ a Boolean valued model such that $\mathcal{M} \vDash s_1$. Since $\exists \vec{v} \varphi(\vec{v}) \in s_1$, $\Qp{\exists \vec{v} \varphi(\vec{v})}_\mathsf{B}^\mathcal{M} = 1_\mathsf{B}$; therefore we can find a maximal antichain $A \subset \mathsf{B}$ and a family $\{\vec{\tau}_a : a \in A\} \subset M$ such that 
\[\Qp{\varphi(\vec{\tau}_a)}_\mathsf{B} \geq a
\]
and
\[
\Qp{\exists \vec{v} \varphi(\vec{v})}_\mathsf{B}^\mathcal{M} = \bigvee_{a \in A} \Qp{\varphi(\vec{\tau}_a)}_\mathsf{B}^\mathcal{M} = 1_\mathsf{B}. 
\]
If we are able to check $\Qp{\theta}_\mathsf{B}^\mathcal{M} \geq a$ for any $a \in A$, we will conclude since 
\[
\Qp{\theta}_\mathsf{B}^\mathcal{M} \geq \bigvee_{a \in A} a =  1_\mathsf{B}.
\]
Consider $a \in A$ and the structure $\mathcal{M}$ together with the assignment $\vec{c} \mapsto \vec{\tau}_a$. 
Consider also the Boolean algebra $\mathsf{B} \upharpoonleft a$. For any $m_1,\ldots,m_n\in\mathcal{M}$ and any $n$-ary relational symbol $R$ of the relational $\omega$-signature $\mathrm{L}$ define
\[
\Qp{R(m_1,\ldots,m_n)}_{\mathsf{B} \upharpoonleft a}^{(\mathcal{M},\vec{c} \ \mapsto \ \vec{\tau}_a)} = a \wedge \Qp{R(m_1,\ldots,m_n)}_\mathsf{B}^\mathcal{M}.
\]
Then one makes $\mathcal{M}$ a $\mathsf{B} \upharpoonleft a$-Boolean valued model for $\mathrm{L}\cup \{\vec{c}\}$ letting 
\[
\Qp{\vartheta[\vec{c} \ \mapsto \ \vec{\tau}_a]}_{\mathsf{B} \upharpoonleft a}^{\mathcal{M}} = a \wedge \Qp{\vartheta[\vec{c} \ \mapsto \ \vec{\tau}_a]}_\mathsf{B}^\mathcal{M}.
\]
In particular the $\mathsf{B} \upharpoonleft a$-value of $\bigwedge s_1 \cup \{\varphi(\vec{c})\}$ in $(\mathcal{M},\vec{c} \mapsto \vec{\tau}_a)$ is $1_{\mathsf{B} \upharpoonleft a} = a$.  Finally, the hypothesis $\vDash_{\mathrm{BVM}} \bigwedge s_1 \cup \{\varphi(\vec{c})\} \rightarrow \theta$ ensures $\Qp{\theta}_{\mathsf{B} \upharpoonleft a}^{(\mathcal{M},\vec{c} \ \mapsto \ \vec{\tau}_a)} = a$ and since 
\[
a = \Qp{\theta}_{\mathsf{B} \upharpoonleft a}^{(\mathcal{M},\vec{c} \ \mapsto \ \vec{\tau}_a)} = a \wedge \Qp{\theta}_\mathsf{B}^\mathcal{M},
\]
we conclude $a \leq \Qp{\theta}_\mathsf{B}^\mathcal{M}$. 
\end{itemize} 
\end{proof}


\section{Forcing notions as consistency properties} \label{sec:for=conprop}

By the results of Section \ref{ForConPro} a consistency property $S$ for $\mathrm{L}_{\kappa \omega}$ can be naturally seen as a forcing notion $\mathbb{P}_S$; then, using the forcing machinery on $\mathbb{P}_S$, we can produce a Boolean valued model with the mixing property of  $\bigwedge p$ for any $p\in S$. In this section we show that it is possible to go the other way round: we prove that any forcing notion $\mathbb{P}$ has a consistency property $S_\mathbb{P}$ associated to it, so that it is equivalent to force with $\mathbb{P}$ or with $\mathbb{P}_{S_\mathbb{P}}$. 

From now on we deal with forcing notions given both by partial orders or by complete Boolean algebras.

Given a complete Boolean algebra  $\mathsf{B}$, 
we show that for some regular $\kappa$ large enough in $V$, $\bool{B}$ is forcing equivalent to a consistency property describing the $\in$-theory of $H_\kappa$ as computed in a $V$-generic extension by 
$\bool{B}$.

\begin{notation} Let $\mathsf{B}$ be a complete Boolean algebra of cardinality $\kappa$.  
The signature $\mathrm{L}$ is $\bp{\in}$. 
$\mathrm{L}(\mathcal{C})_{\infty \omega}$ is produced by the set of constants $\mathcal{C} = V^\mathsf{B} \cap H_{\kappa^+}$. We use $\phi^{H_{\kappa^+}}$ to denote that all quantifiers from $\phi$ are restricted to $H_{\kappa^+}$.
We write $\Qp{\phi}_\bool{B}$ rather than $\Qp{\phi}^{V^{\bool{B}}}_\bool{B}$.
\end{notation} 	

\begin{theorem}\label{thm:equivforcconsprop}
For any complete Boolean algebra $\mathsf{B}$ of size less or equal than $\kappa$ and any regular cardinal $\lambda$ the following holds:
\begin{enumerate}[label=(\roman*)]
\item $S_\mathsf{B} = \{s \in [\mathrm{L}(\mathcal{C})_{\lambda \omega}]^{< \omega}: \Qp{(\bigwedge s)^{H_{\check{\kappa}^+}}}_\mathsf{B} > 0_\mathsf{B}\}$ is a consistency property,
\item the map 
\begin{align*}
 \pi_{\mathsf{B}} : (S_\mathsf{B}, \leq) &\rightarrow (\mathsf{B}^+,\leq_\mathsf{B}) \\ s &\mapsto \Qp{(\bigwedge s)^{H_{\check{\kappa}^+}}}_\mathsf{B} 
 \end{align*}
 is a dense embedding. 
In particular $\mathsf{B}$ and $S_{\mathsf{B}}$ are equivalent forcing notions.
\end{enumerate}
\end{theorem}

\begin{proof}

We first prove $(ii)$. 
\begin{itemize}
\item If $p \leq q$, then $q \subseteq p$ and $\pi(p) = \Qp{\bigwedge p}_\mathsf{B} \leq \Qp{\bigwedge q}_\mathsf{B} = \pi(q)$.
\item We have $p \perp q \Leftrightarrow p \cup q \notin S_\mathsf{B} \Leftrightarrow \Qp{\bigwedge (p \cup q)}_\mathsf{B} = 0_\mathsf{B} \Leftrightarrow \Qp{\bigwedge p}_\mathsf{B} \wedge \Qp{\bigwedge q}_\mathsf{B} = 0_\mathsf{B} \Leftrightarrow \pi(p) \perp \pi(q)$.
\item Let $\dot{G}=\bp{(\check{b},b):\,b\in\bool{B}}$ be the canonical $\mathsf{B}$-name for a $V$-generic filter. Since for any $b \in \mathsf{B}^+$ the $\mathsf{B}$-value of $\check{b} \in \dot{G}$ is $b$, the map $\pi$ is surjective and in particular $\pi[S_\mathsf{B}]$ is dense in $\mathsf{B}^+$.  
\end{itemize}

Now we prove $(i)$. We have to check that $S_\bool{B}$ satisfies the clauses of Def. \ref{def:ConProInf}.
Note that by choice of $\kappa$, 
\[
\Qp{\forall v\in H_{\check{\kappa}^+}\phi^{H_{\check{\kappa}^+}}(v)}^{V^{\bool{B}}}_{\bool{B}}=\bigwedge_{\tau\in \mathcal{C}}\Qp{\phi^{H_{\check{\kappa}^+}}(v)}^{V^{\bool{B}}}_{\bool{B}}.
\]
In view of the above observation, for notational simplicity we use $\Qp{\phi}_\mathsf{B}$ instead of $\Qp{\phi^{H_{\check{\kappa}^+}}}^{V^{\bool{B}}}_\mathsf{B}$. Note also that in the proof below 
we will only be interested in formulae where quantifiers range over 
(and constants belong to) $H_{\kappa^+}\cap V^\bool{B}$.

\begin{description}
\item[(Con)] Consider $s \in S_\mathsf{B}$ and $\phi \in \mathrm{L}(\mathcal{C})_{\infty\omega}$. If $\phi$ and $\neg \phi$ are both in $s$, $\Qp{\bigwedge s}_\mathsf{B} \leq \Qp{\phi \wedge \neg \phi}_\mathsf{B} = 0_\mathsf{B}$, a contradiction since $\Qp{\bigwedge s}_\mathsf{B} > 0_\mathsf{B}$. Then for any $\phi$, either $\phi \notin s$ or $\neg \phi \notin s$.
\item[(Ind.1)] Consider $s \in S_\mathsf{B}$ and $\neg \phi \in s$. Since $\Qp{\neg \phi}_\mathsf{B} = \Qp{\phi \neg}_\mathsf{B}$, $\Qp{\bigwedge (s \cup \{\phi \neg\})}_\mathsf{B} = \Qp{\bigwedge s}_\mathsf{B} > 0_\mathsf{B}$ and $s \cup \{\phi \neg\} \in S_\mathsf{B}$.
\item[(Ind.2)] Consider $s \in S_\mathsf{B}$ and $\bigwedge \Phi \in s$. For any $\phi \in \Phi$, $\Qp{\bigwedge (s \cup \{\phi\})}_\mathsf{B} = \Qp{\bigwedge s}_\mathsf{B} > 0_\mathsf{B}$ and $s \cup \{\phi\} \in S_\mathsf{B}$.
\item[(Ind.3)] Consider $s \in S_\mathsf{B}$, $\forall v \phi(v) \in s$ and $\tau \in \mathcal{C}$. We have 
\[
\Qp{\forall v \phi(v)}_\mathsf{B} = \bigwedge_{\sigma \in V^\mathsf{B} \cap H_{\kappa^+}} \Qp{\phi(\sigma)}_\mathsf{B} \leq \Qp{\phi(\tau)}_\mathsf{B}.
\] 
Therefore 
\[
\Qp{\bigwedge (s \cup \{\phi(\tau)\})}_\mathsf{B} = \Qp{\bigwedge s}_\mathsf{B} > 0_\mathsf{B},
\]
and $s \cup \{\phi(\tau)\} \in S_\mathsf{B}$.
\item[(Ind.4)] Consider $s \in S_\mathsf{B}$ and $\bigvee \Phi \in s$. Suppose that for no $\phi \in \Phi$, $s \cup \{\phi\} \in S_\mathsf{B}$. Then for any $\phi \in \Phi$, $\Qp{\bigwedge (s \cup \{\phi\})}_\mathsf{B} = \Qp{\bigwedge s}_\mathsf{B} \wedge \Qp{\phi}_\mathsf{B} = 0_\mathsf{B}$. Therefore $\Qp{\bigwedge s}_\mathsf{B} \leq \Qp{\neg \phi}_\mathsf{B}$ for any $\phi \in \Phi$. Since $\bigwedge_{\phi \in \Phi}\Qp{\neg \phi}_\mathsf{B}$ is the greatest lower bound  of $\{\Qp{\neg \phi}_\mathsf{B} : \phi \in \Phi\}$, we have $\Qp{\bigwedge s}_\mathsf{B} \leq \bigwedge_{\phi \in \Phi}\Qp{\neg \phi}_\mathsf{B} = \Qp{\neg \bigvee \Phi}_\mathsf{B}$. Then $\Qp{\bigwedge (s \cup \{\neg \bigvee \Phi\})}_\mathsf{B} = \Qp{\bigwedge s}_\mathsf{B} > 0_\mathsf{B}$, but since $\bigvee \Phi$ and $\neg \bigvee \Phi$ are both in $s \cup \{\neg \bigvee \Phi\}$, $\Qp{\bigwedge (s \cup \{\neg \bigvee \Phi\})}_\mathsf{B} = 0_\mathsf{B}$, a contradiction. 
\item[(Ind.5)] Consider $s \in S_\mathsf{B}$ and $\exists v \phi(v) \in s$. Suppose that for no $\tau\in\mathcal{C}$, $s \cup \{\phi(\tau)\} \in S_\mathsf{B}$. Then for any $\tau\in\mathcal{C}$, 
\[
\Qp{\bigwedge (s \cup \{\phi(\tau)\})}_\mathsf{B} = \Qp{\bigwedge s}_\mathsf{B} \wedge \Qp{\phi(\tau)}_\mathsf{B} = 0_\mathsf{B}.
\]
This gives that
\[ 
\Qp{\bigwedge s}_\mathsf{B} \leq \Qp{\neg \phi(\tau)}_\mathsf{B}
\]
for any $\tau \in V^\mathsf{B} \cap H_{\kappa^+}$. Therefore
\[
\Qp{\bigwedge s}_\bool{B} \leq \bigwedge_{\tau \in V^\mathsf{B} \cap H_{\kappa^+}} \Qp{\neg \phi(\tau)}_\mathsf{B} = \Qp{\forall v \neg \phi(v)}_\mathsf{B} = \Qp{\neg \exists v \phi(v)}_\mathsf{B}.
\]
Note that the equality 
\[
\bigwedge_{\tau \in V^\mathsf{B} \cap H_{\kappa^+}} \Qp{\neg \phi(\tau)}_\mathsf{B} = \Qp{\forall v \neg \phi(v)}_\mathsf{B}
\]
only holds because the quantifiers from $\phi$ are restricted to $H_{\check{\kappa}^+}$. Therefore 
\[
\Qp{\bigwedge (s \cup \{\neg \exists v \phi(v)\})}_\mathsf{B} = \Qp{\bigwedge s}_{\bool{B}} > 0_\mathsf{B}.
\]
But now $\exists v \phi(v)$ and $\neg \exists v \phi(v)$ are both in $s \cup \{\neg \exists v \phi(v)\}$, hence
\[
\Qp{\bigwedge (s \cup \{\neg \exists v \phi(v)\})}_\mathsf{B} = 0_\mathsf{B}.
\]
We reached a contradiction. 
\item[(Str.1)] Suppose $s \in S_\mathsf{B}$ and $\tau = \sigma \in s$; since $\mathsf{B}$-valued models for set theory verify $\Qp{\tau = \sigma}_\mathsf{B} = \Qp{\sigma = \tau}_\mathsf{B}$,  $\Qp{\bigwedge (s \cup \{\sigma = \tau\})}_\mathsf{B} = \Qp{\bigwedge s}_\mathsf{B} > 0_\mathsf{B}$ and $s \cup \{\sigma = \tau\} \in S_\mathsf{B}$.
\item[(Str.2)] Suppose $s \in S_\mathsf{B}$ and $\{\sigma = \tau, \phi(\tau)\} \subset s$. We have $\Qp{\sigma = \tau}_\mathsf{B} \wedge \Qp{\phi(\tau)}_\mathsf{B} \leq \Qp{\phi(\sigma)}_\mathsf{B}$; therefore $\Qp{\bigwedge (s \cup \{\phi(\sigma)\})}_\mathsf{B} > 0_\mathsf{B}$ and $s \cup \{\phi(\sigma)\} \in S_\mathsf{B}$.
\item[(Str.3)] Trivial since $\mathrm{L}=\bp{\in}$ has no constant symbol.
\end{description}

\end{proof}

Note that the only formulae one needs to keep in $S_\mathsf{B}$ in order to ensure that there is a dense embedding between both forcing notions are $\check{b} \in \dot{G}$. This is because in order to prove that the embedding has dense image, one only uses that the $\mathsf{B}$-value of $\check{b} \in \dot{G}$ is $b$. In particular one can consider various choices of constants $\mathcal{C}$ to produce the desired consistency property $S_\bool{B}$, other than the one we made.

\section{Appendix}\label{sec:app}

We collect here some results which are useful to clarify several concepts but not central.

\subsection{Separating Tarski satisfiability from boolean satisfiability}

In this section we show that being satisfiable in the ordinary sense (e.g. with respect to Tarski semantics)
is strictly stronger than being boolean satisfiable, which is also strictly stronger than being weakly boolean satisfiable.

We first show that there is a boolean satisfiable theory which has no Tarski model.
\begin{fact}\label{fac:tarskiinc}
Let $\mathrm{L}=\bp{F}\cup\bp{d_n:n\in\omega}\cup\bp{e_\alpha:\alpha<\omega_1^V}$ with $F$ a binary predicate.
Consider the $\mathrm{L}_{\omega_2^V\omega}$-theory $S$ given by 
\[
\forall x\,\exists! y\, F(x,y)
\]
\[
\exists y\, F(x,y)\leftrightarrow \bigvee_{n\in\omega}x=d_n
\]
\[
\bigvee_{n\in\omega} F(d_n,e_\alpha)
\]
for all  $\alpha<\omega_1^V$.

Then every countable fragment of $S$ has a Tarski model, while $S$ has no Tarski model.

Furthermore $S$ has a boolean valued model.
\end{fact}
Note that $S$ witnesses the failure of the compactness and completeness theorems for Tarski semantics for 
$\mathrm{L}_{\omega_2^V\omega}$; it is clearly a counterexample to compactness for this semantics; it is also a counterexample to completeness (using the axiom system we present in Section \ref{subsec:gentzencalc}) since $S\not\vdash\emptyset$ in view of
Thm. \ref{them:boolcompl}.
\begin{proof}
Given a countable fragment $R$ of $S$, find $\beta$ countable and such that any $e_\alpha$ occurring in some formula in $R$ has $\alpha<\beta$. Then $(\beta,f)$ where $f$ is a surjection of $\omega$ onto $\beta$ can be extended to a model of $R$ by mapping $d_n$ to $n$ and $e_\alpha$ to $\alpha$ for any $n\in\omega,\alpha<\beta$.

Note that the interpretation of $F$ in any Tarski model of $S$ in $V$ is a map with domain a countable set and range an uncountable set in $V$. Hence no such model can exist in $V$.

Now if $G$ is $\Coll(\omega,\omega_1)$-generic in $V[G]$ the generic function $\omega\to\omega_1^V$ given by $\cup G$ gives in $V[G]$ a Tarski model of 
$S$. Taking this into account, in $V$ consider the $\RO(\Coll(\omega,\omega_1^V))$-valued model 
$\mathcal{M}=(\omega_1^V,F^\mathcal{M},d_n^\mathcal{M}:n\in\omega,
e_\alpha^\mathcal{M}:\alpha<\omega_1^V)$ given by 
\begin{itemize}
\item
$R^\mathcal{M}(n,\alpha)=\Reg{\bp{q\in\Coll(\omega,\omega_1^V):\, \ap{n,\alpha}\in q}}$ for $n\in\omega$ and $\alpha<\omega_1^V$; $R^\mathcal{M}(\beta,\alpha)=0_{\Coll(\omega,\omega_1^V)}$ for $\beta\not\in\omega$ and $\alpha<\omega_1^V$;
\item
$d_n^\mathcal{M}=n$ for all $n\in\omega$,
\item
$e_\alpha^\mathcal{M}=\alpha$ for all $\alpha\in\omega_1^V$.
\end{itemize}
It can be checked that in $V$ it holds that 
$\mathcal{M}$ assigns value $1_ {\Coll(\omega,\omega_1^V)}$ to all axioms of $S$.
\end{proof}

Now we exhibit a theory $T$ which is a counterexample to the compactness theorem with respect to boolean satisfiability: all finite fragments of $T$ are boolean satisfiable while $T$ is not.

\begin{fact}
Consider the first order $\mathrm{L}_{\omega_1\omega}$-theory $T$ for $\mathrm{L}=\bp{d_n:n\in\omega,c_m:\,m\in\omega}$ with axioms:
\begin{itemize}
\item
$\bigwedge_{n\in\omega}\bigvee_{m\in \omega}d_n= c_m$,
\item
$\bigwedge_{n\neq m\in\omega}c_n\neq c_m$,
\item
$d_n\neq c_m$
for $n,m\in \omega$.
\end{itemize}
The following holds:
\begin{itemize}
\item
Every finite fragment of $T$ is Tarski satisfiable.
\item
$T$ is not boolean satisfiable.
\item
$T$ is weakly boolean satisfiable.
\end{itemize}
\end{fact}

\begin{proof}
Let:
\begin{itemize}
\item
$\bool{B}$ be the boolean completion of the Cohen forcing $\omega^{<\omega}$,
\item
$M=\bp{\sigma\in V^{\bool{B}}: \Qp{\sigma\in\check{\omega}}=1_{\bool{B}}}$,
\item
$\dot{r}$ be the canonical $\bool{B}$-name for the Cohen generic real.
\item
$\mathcal{M}$ be the $\bool{B}$-model for $\mathrm{L}$ with domain $M$,
$\Qp{\cdot=\cdot}^{\mathcal{M}}= \Qp{\cdot=\cdot}^{V^{\bool{B}}}$,
and interpretation of $d_n$ by $\dot{r}(\check{n})$ and $c_m$ by $\check{m}$.
\end{itemize}
Then $\mathcal{M}$ witnesses that $T$ is weakly boolean satisfiable.

$T$ cannot be boolean satisfiable because in any boolean valued model it cannot be that
$d_n\neq c_m$ gets boolean value $1_\bool{B}$ for all $m$ while
also $\bigvee_{m\in \omega}d_n= c_m$ gets the same value.

$(\omega,c_n\mapsto n:n\in\omega)$ can be extended to a Tarski model of any finite fragment of $T$.
\end{proof}
Our last example is a \emph{finite} weakly boolean satisfiable theory which is not boolean
satisfiable.

\begin{fact}
Consider the finite $\mathrm{L}_{\omega\omega}$-theory $T$ for 
$\mathrm{L}=\bp{d,c_0,c_1}$ with axioms:
\begin{itemize}
\item
$\bigvee_{m\in 2}d= c_m$,
\item
$c_0\neq c_1$,
\item
$d\neq c_m$
for $m\in 2$.
\end{itemize}
The following holds:
\begin{itemize}
\item
$T$ is not boolean satisfiable.
\item
$T$ is weakly boolean satisfiable.
\end{itemize}
\end{fact}
\begin{proof}
Let $\bool{B}=\bp{0,a,\neg a,1}$ be the four elements boolean algebra, let 
$\mathcal{M}$ consists of the four possible functions $f:\bp{a,\neg a}\to 2$.
Let $c_i$ be interpreted by the constant functions with value $i$ and $d$ by one of the other two.
Set $\Qp{f=g}^{\mathcal{M}}_\bool{B}=\bigvee \{ b\in\bp{a,\neg a}: f(b)=g(b) \} $.
Then
\[
\Qp{\bigvee_{m\in 2}d= c_m}^{\mathcal{M}}_\bool{B}=
\Qp{c_0\neq c_1}^{\mathcal{M}}_\bool{B}=1_\bool{B}
\] 
and $\Qp{d\neq c_i}^{\mathcal{M}}_\bool{B}>0_\bool{B}$ for both $i=0,1$.
Hence $T$ is weakly boolean satisfiable.

$T$ cannot be boolean satisfiable since 
\[
\Qp{d\neq c_0}^{\mathcal{N}}_{\bool{C}}=
\neg\Qp{d=c_1}^{\mathcal{N}}_{\bool{C}}
\] 
in all $\bool{C}$-valued models $\mathcal{N}$ of 
$c_0\neq c_1\wedge \bigvee_{m\in 2}d= c_m$. Hence we cannot have that
 \[
 \Qp{d\neq c_0}^{\mathcal{N}}_{\bool{C}}= \Qp{d_n\neq c_1}^{\mathcal{N}}_{\bool{C}}=1_{\bool{C}}
 \] 
 in any boolean valued model model of the other axioms of $T$.
\end{proof}

We conclude this part noting that boolean satisfiability is the correct generalization to $\mathrm{L}_{\infty\omega}$ of Tarski satisfiability:
\begin{fact}
Assume $T$ is a first order theory. Then $T$ is boolean
satisfiable if and only if $T$ is Tarski satisfiable.
\end{fact}
\begin{proof}
By Thm. \ref{them:boolcompl} any boolean satisfiable first order theory 
is realized in a $\bool{B}$-valued model
$\mathcal{M}$ with the mixing property. If $G$ is a ultrafilter on $\bool{B}$, $\mathcal{M}/_G$ models $T$ by Proposition \ref{prop:mixfull} and Thm. \ref{thm:fullLos}.
\end{proof}



\subsection{Proof of Fact \ref{fac:pressubslambdaanyform} and Proposition \ref{prop:mixfull}}\label{subsec:mixfull}

We first prove Fact \ref{fac:pressubslambdaanyform}.
\begin{proof}
We proceed by induction on the complexity of $\phi(x_i:i<\alpha)$. The Fact holds by definition for atomic formulae.
Assume the Fact for all proper subformulae of $\phi(x_i:i<\alpha)$.

Now note that the desired inequality entails that for all $\beta$ and 
$(\sigma_i:i<\beta)$, $(\tau_i:i<\beta)$ in
$\mathcal{M}^\beta$ and all $\psi(x_i:i<\beta)$ proper subformula of $\phi(x_i:i<\alpha)$,
\[
\bigg(\bigwedge_{i\in\beta}\Qp{\tau_i=\sigma_i}_\mathsf{B} \bigg) \wedge \Qp{\psi(\tau_i:\,i<\beta)}_\mathsf{B} = \bigg(\bigwedge_{i\in\beta}\Qp{\tau_i=\sigma_i}_\mathsf{B} \bigg) \wedge \Qp{\psi(\sigma_i:\,i<\beta)}_\mathsf{B}.
\]
 Now if $\phi=\neg\psi$ the above equality is extended to $\phi$. It is also preserved if $\phi=\bigvee \Phi$ since
 \begin{align*}
\bigg(\bigwedge_{i\in\alpha}\Qp{\tau_i=\sigma_i}_\mathsf{B} \bigg) \wedge \Qp{\bigvee\Phi(\tau_i:\,i<\alpha)}_\mathsf{B} = 
\\
=\bigg(\bigwedge_{i\in\alpha}\Qp{\tau_i=\sigma_i}_\mathsf{B} \bigg) \wedge \bigvee_{\psi\in \Phi}\Qp{\psi(\tau_i:\,i<\alpha)}_\mathsf{B} =
\\
=\bigvee_{\psi\in\Phi}\bigg(\bigwedge_{i\in\alpha}\Qp{\tau_i=\sigma_i}_\mathsf{B}  \wedge \Qp{\psi(\tau_i:\,i<\alpha)}_\mathsf{B} \bigg) =
\\
=\bigvee_{\psi\in\Phi}\bigg(\bigwedge_{i\in\alpha}\Qp{\tau_i=\sigma_i}_\mathsf{B}  \wedge \Qp{\psi(\sigma_i:\,i<\alpha)}_\mathsf{B} \bigg) =
\\
=\bigg(\bigwedge_{i\in\alpha}\Qp{\tau_i=\sigma_i}_\mathsf{B} \bigg) \wedge \bigvee_{\phi\in \Phi}\Qp{\psi(\sigma_i:\,i<\alpha)}_\mathsf{B} =
\\
=\bigg(\bigwedge_{i\in\alpha}\Qp{\tau_i=\sigma_i}_\mathsf{B} \bigg) \wedge \Qp{\bigvee\Phi(\sigma_i:\,i<\alpha)}_\mathsf{B}.
\end{align*}
Similarly
\begin{align*}
\bigg(\bigwedge_{i\in\alpha}\Qp{\tau_i=\sigma_i}_\mathsf{B} \bigg) \wedge \Qp{\exists (y_j:j\in\beta)\psi(\tau_i:\,i<\alpha,y_j:j\in\beta)}_\mathsf{B} = 
\\
=\bigg(\bigwedge_{i\in\alpha}\Qp{\tau_i=\sigma_i}_\mathsf{B} \bigg) \wedge \bigvee_{(\eta_j:j\in\beta)\in\mathcal{M}^\beta }\Qp{\psi(\tau_i:\,i<\alpha,\eta_j:j\in\beta)}_\mathsf{B} =
\\
=\bigvee_{(\eta_j:j\in\beta)\in\mathcal{M}^\beta }\bigg(\bigwedge_{i\in\alpha}\Qp{\tau_i=\sigma_i}_\mathsf{B}  \wedge \Qp{\psi(\tau_i:\,i<\alpha,\eta_j:j\in\beta)}_\mathsf{B} \bigg) =
\\
=\bigvee_{(\eta_j:j\in\beta)\in\mathcal{M}^\beta }\bigg(\bigwedge_{i\in\alpha}\Qp{\tau_i=\sigma_i}_\mathsf{B}  \wedge \Qp{\psi(\sigma_i:\,i<\alpha,\eta_j:j\in\beta)}_\mathsf{B} \bigg) =
\\
=\bigg(\bigwedge_{i\in\alpha}\Qp{\tau_i=\sigma_i}_\mathsf{B} \bigg) \wedge \bigvee_{(\eta_j:j\in\beta)\in\mathcal{M}^\beta }\Qp{\psi(\sigma_i:\,i<\alpha,\eta_j:j\in\beta)}_\mathsf{B} =
\\
=\bigg(\bigwedge_{i\in\alpha}\Qp{\tau_i=\sigma_i}_\mathsf{B} \bigg) \wedge \Qp{\exists (y_j:j\in\beta)\psi(\sigma_i:\,i<\alpha,y_j:j\in\beta)}_\mathsf{B}.
\end{align*}
The cases of $\bigwedge, \forall$ are handled similarly.
\end{proof}


We can now prove Proposition \ref{prop:mixfull}.
\begin{proof}
Let $\exists \overline{v} \phi(\overline{v})$ be a $\mathrm{L}_{\infty \infty}$-sentence. Fix a maximal antichain $A$ among 
\[
\bp{b \in \bool{B} : b \leq \Qp{\phi(\overline{c_b})} \text{ for some } \overline{c} \in M^{|\overline{v}|}}.
\]
For each $b \in A$ let $\overline{c_b} = (c_{i,b} : i \in I)$. The mixing property in $\mathcal{M}$ gives $c_i$ for each $i \in I$ such that $\Qp{c_i=c_{i,b}}_\mathsf{B} \geq b$ for all $b \in A$. Let $\overline{c} = (c_i : i \in I)$. Then 
\begin{gather*}
\Qp{\exists \overline{v} \phi(\overline{v})}= \bigvee A = \bigvee_{b \in A} \Qp{\phi(\overline{c_b})} =\bigvee_{b \in A} (b \wedge \Qp{\phi(\overline{c_b})} \wedge \bigwedge_{i\in I} \Qp{c_i=c_{i,b}}_\mathsf{B}) \leq \\
\bigvee_{b \in A} (b \wedge \Qp{\phi(\overline{c})}_\mathsf{B})= \Qp{\phi(\overline{c})}_\mathsf{B}.
\end{gather*}
\end{proof}

%
%

\section*{Concluding remarks}
Mansfield's completeness theorem follows from his proof that if $S$ is a consistency property for 
$\mathrm{L}_{\infty\infty}$, there is a $\RO(\mathbb{P}_S)$-valued model $\mathcal{M}_S$
such that $\Qp{\bigwedge s}=\Reg{\bp{s}}$ for any $s\in S$. However there is no reason to expect that the model $\mathcal{M}_S$ produced in Mansfield's proof is full.
We conjecture it is not, at least for some $S$.

We also conjecture that if $S$ is a consistency property for $\mathrm{L}_{\infty\infty}$, 
our model $\mathcal{A}_S$ may not satisfy $\Qp{\psi}=\Reg{\bp{r\in S:\psi\in r}}$ for some
formula $\psi$ of $\mathrm{L}_{\infty\infty}$.
The key point is that the $\mathrm{L}_{\infty\infty}$-semantics of existential quantifiers over infinite strings is not forcing invariant: if one forces the addition of a new countable sequence to some $\bool{B}$-valued model $\mathcal{M}$ in $V$, it may be the case that $\exists\vec{v}\psi$ gets Boolean value $0_\bool{B}$ in  $V$ and positive $\RO^{V[G]}(\bool{B})$-Boolean value in the generic extension $V[G]$. This makes our proof of Thm. \ref{GenFilThe}  break down when handling the existential quantifier clause for $\mathrm{L}_{\infty\infty}$ over an infinite string.

We dare the following:
\begin{conjecture}
Assume $\mathrm{L}$ is a relational $\omega$-signature.
There are Boolean satisfiable $\mathrm{L}_{\infty\infty}$-theories which do not have a Boolean valued 
model with the mixing property.
\end{conjecture}

Another point to be clarified on the completeness of Boolean valued semantics for $\mathrm{L}_{\infty\infty}$ is the following:

\begin{question}
Assume $\mathrm{L}$ is a relational $\lambda$-signature for $\lambda>\omega$. Does the completeness theorem for consistent
(according to the $\mathrm{L}_{\infty\infty}$-Gentzen's calculus)
 $\mathrm{L}_{\infty\infty}$-theories holds?
\end{question}
Mansfield's model existence theorem does not apply to such theories as the proof of (\ref{eqn:subslambda}) for the model obtained in Mansfield's proof breaks down for the obvious modification of the notion of consistency property required in order to deal with Boolean valued models for $\lambda$-signatures (e.g. one should replace clause (Str.2) of a consistency property with the stronger: \emph{``If $\bp{\bigwedge_{i\in I}c_i=d_i,\phi(c_i:i\in I)}\in s\in S$, then $\bp{\phi(d_i:i\in I)}\cup s\in S$''}.)  In the tentative proof of (\ref{eqn:subslambda}) for the model obtained in Mansfield's proof one should replace our argument with one requiring that a distributivity law for infinite conjunctions of infinite disjunctions holds.
The latter may not hold for $\bool{B}_S$.

Note finally that while Boolean compactness fails for $\mathrm{L}_{\infty\omega}$, one can prove the following curious form of weak Boolean  compactness:
\begin{fact}
Assume $T$ is a family of Boolean satisfiable $\mathrm{L}_{\infty\omega}$-sentences.
Then $T$ is weakly Boolean satisfiable.
\end{fact}
This holds noticing that if $S_\psi$ is a consistency property that witnesses that $\psi$ is Boolean consistent, 
$S=\bigcup_{\psi\in T}S_\psi$ is a consistency property  such that $\mathcal{A}_S$ assigns a positive Boolean value to any $\psi\in T$.

\bibliographystyle{plain}
	\bibliography{Biblio}

\end{document}